\documentclass[11pt, a4paper]{amsart}

\usepackage{graphicx}
\usepackage{epsfig}
\usepackage{amsmath, amsthm, amssymb, mathtools}
\usepackage{enumitem}
\usepackage[mathscr]{euscript}
\usepackage[english]{babel}
\usepackage{aliascnt}
\usepackage[all,cmtip]{xy}
\usepackage{url}
\usepackage{tikz-cd}
\usepackage{hyperref}
\usepackage[capitalize]{cleveref}

\def\ds{\displaystyle}

\def\Hom{\mathrm{Hom}}

\def\Quad{\mathrm{Quad}}

\def\Bij{\mathrm{Bij}}
\def\Map{\mathrm{Map}}
\def\Or{\mathrm{Or}}
\def\End{\mathrm{End}}

\def\det{\mathrm{det}}
\def\norm{\mathrm{Nm}}
\def\sgn{\mathrm{sgn}}
\def\trace{\mathrm{Tr}}

\def\id{\mathrm{id}}

\def\Q{\mathbb{Q}}
\def\N{\mathbb{N}}
\def\Z{\mathbb{Z}}

\newcommand{\fundgroup}[2][]{\pi_{#2}}
\newcommand{\algs}{\operatorname{-\mathbf{Alg}}}

\newcommand{\fingsets}{\operatorname{-\mathbf{set}}}
\newcommand{\finetalgs}{\operatorname{-\mathbf{\acute{e}t}}}

\newcommand{\extpower}{\smash{\textstyle\bigwedge}}
\renewcommand{\subset}{\subseteq}

\newcommand{\ferrand}{\Phi}

\def\simto{\ds\mathop{\longrightarrow}^\sim\,}

\def\into{\hookrightarrow}
\def\onto{\twoheadrightarrow}

\newcommand{\set}[2][1]{\{#1,\dots,#2\}}

\newcommand{\fixpower}[4][]{({#2}^{\otimes_{#1} #3})^{#4}}
\newcommand{\conjugate}[2]{#1^{(#2)}}

\let\sect\S
\renewcommand\S[1]{\mathrm{S}_{#1}}
\newcommand\A[1]{\mathrm{A}_{#1}}
\newcommand\Abar[1]{\overline{\mathrm{A}}_{#1}}

\newcommand{\disc}[1]{\delta_{#1}}
\newcommand{\discalg}[1]{\Delta_{#1}}
\newcommand{\discalgs}[2]{\smash{\Delta_{#2}^{\mathrm{#1}}}}
\newcommand{\stdi}[1]{\tau_{#1}}

\def\midotimes{\bigotimes}

\newtheorem{mainthm}{Theorem}
\crefname{mainthm}{Theorem}{Theorems}
\crefname{enumi}{}{}

\newcounter{nootje}
\newtheorem{theorem}{Theorem}[section] 
\crefname{theorem}{Theorem}{Theorems}

\newtheorem{lemma}[theorem]{Lemma}
\crefname{lemma}{Lemma}{Lemmas}

\newtheorem{proposition}[theorem]{Proposition}
\crefname{proposition}{Proposition}{Propositions}

\newtheorem{corollary}[theorem]{Corollary}
\crefname{corollary}{Corollary}{Corollaries}

\theoremstyle{definition} 

\newtheorem{definition}[theorem]{Definition}
\crefname{definition}{Definition}{Definitions}

\newtheorem{remark}[theorem]{Remark}
\crefname{remark}{Remark}{Remarks}

\newtheorem{example}[theorem]{Example}
\crefname{example}{Example}{Examples}

\begin{document}

\author{Owen Biesel}
\address{Owen Biesel\\
Mathematisch Instituut\\
Niels Bohrweg 1\\
Leiden University\\
The Netherlands.}
\email{bieselod@math.leidenuniv.nl}

\author{Alberto Gioia}
\email{alberto.gioia@univie.ac.at}

\title{A new discriminant algebra construction}
\begin{abstract}
A discriminant algebra operation sends a commutative ring $R$ and an $R$-algebra $A$ of rank $n$ to an $R$-algebra $\Delta_{A/R}$ of rank $2$ with the same discriminant bilinear form. 
Constructions of discriminant algebra operations have been put forward by Rost, Deligne, and Loos. 
We present a simpler and more explicit construction that does not break down into cases based on the parity of $n$. 
We then prove properties of this construction, and compute some examples explicitly.
\end{abstract}

\date{March 23, 2016}

\maketitle
\tableofcontents

\section{Introduction}

An $n$-fold covering map of topological spaces $X\to S$ has an associated $2$-fold ``orientation'' covering, which on fibers is described by the set of orderings of the $n$-element fiber of $X$, up to reorderings by even permutations.
The algebraic analogue producing a rank-$2$ \'etale algebra from a rank-$n$ one is called the \emph{discriminant algebra}, and we briefly review its classical constructions below.

Let $f(x)$ in $\mathbb{Q}[x]$ be an irreducible polynomial of degree $n$. 
Suppose the Galois group of $f$ is the symmetric group $\S{n}$, and let $x_1,\ldots, x_n$ be the roots of $f$ in its splitting field $N$. 
The subfield of $N$ consisting of elements invariant under even permutations of those roots is called the \emph{discriminant field} of $L\coloneqq \Q[x]/(f(x))$. 
It is formed by adjoining a square root of the discriminant of $L$ over $K$, hence the name.

Following \cite[\sect 18]{InvolBook}, we may extend this operation to general rank-$n$ separable algebras $A$ over a field $K$ of arbitrary characteristic: such an algebra corresponds to an $n$-element $\fundgroup{K}$-set $X$, where $\fundgroup{K}$ is the absolute Galois group of $K$, and the discriminant \emph{algebra} of $A$ is the rank-$2$ separable algebra $\discalg{A/K}$ corresponding to the $2$-element set of \emph{orientations} of $X$, orderings of $X$ up to re-orderings by even permutations.
(If the characteristic of $K$ is not $2$, then the discriminant algebra of $A$ may again be presented as $K[x]/(x^2-d)$, where $d\in K$ is the discriminant of $A$ with respect to some $K$-basis; see \cite[Proposition 18.24]{InvolBook}.)
The same $\pi$-set construction applies more generally whenever $R$ is a connected ring and $A$ is a rank-$n$ projective separable $R$-algebra, since there is still a contravariant equivalence between projective separable $R$-algebras and finite sets equipped with an action by a suitable profinite group $\fundgroup{R}$.
In \cite{Wat87}, William Waterhouse drops the connectedness hypothesis by interpreting \'etale rank-$n$ algebras as $\S{n}$-torsors, and the discriminant algebra mapping from $\S{n}$-torsors to $\S{2}$-torsors is merely the one coming from the sign homomorphism $\S{n}\to\S{2}$.

Our goal in this paper is to extend the discriminant algebra construction even further to the case of $R$ a ring and $A$ a rank-$n$ projective $R$-algebra, with no separability hypothesis.
(In this paper, all rings and algebras are commutative, associative, and unital, so an $R$-algebra structure on $A$ is merely a homomorphism of commutative rings $R\to A$.)
The main feature we wish to keep is that the discriminant algebra of $A$ should have the same discriminant as that of $A$; this is easy to check in the original motivating case of a finite separable algebra $A$ over $K$ in characteristic other than $2$:  
If $A$ has discriminant $d$ with respect to some $K$-basis, then the discriminant algebra for $A$ over $K$ is $\Delta\cong K[x]/(x^2-d)$.
Since $2$ is a unit, we may as well write $\Delta\cong K[x]/(x^2-d/4)$, which, with respect to the $K$-basis $\{1,x\}$, also has discriminant
\[\det\begin{pmatrix}
 \trace_{\Delta/K}(1) & \trace_{\Delta/K}(x)\\
 \trace_{\Delta/K}(x) & \trace_{\Delta/K}(x^2)
\end{pmatrix} = \begin{pmatrix}
 2 & 0\\
 0 & d/2
\end{pmatrix} = d.\]

In this paper, we put forth such a general discriminant algebra construction.
In the setting of a ring $R$ and algebra $A$ that is locally free of rank $n$ as an $R$-module, it still makes sense to talk about the trace and discriminant maps: $A$ has a \emph{discriminant bilinear form} $\disc{A/R}$ on $\extpower^n A$ defined by
\[
\disc{A/R}(a_1\wedge\cdots\wedge a_n, b_1\wedge\cdots\wedge b_n) = \det\bigl(\trace_{A/R}(a_ib_j)\bigr)_{ij}.
\]
Here is the first main result of this paper:

\begin{mainthm}[proven as \cref{exact-sequence}]\label{mainthm-identify-discriminants}
 There is an assignment $(R,A)\mapsto \discalg{A/R}$ sending any rank-$n$ algebra $A$ over any ring $R$ to a rank-$2$ $R$-algebra $\discalg{A/R}$ that is endowed with a canonical isomorphism $\extpower^2 \discalg{A/R}\cong\extpower^n A$ identifying the discriminant bilinear form of $\discalg{A/R}$ with that of $A$.
\end{mainthm}
We call $\discalg{A/R}$ the \emph{discriminant algebra} of $A$ over $R$.
In a 2005 letter to Manjul Bhargava and Markus Rost, Pierre Deligne suggested a list of four other properties a discriminant algebra operation $(R,A)\mapsto \discalg{A/R}$ should have; see \cite{DeligneLett} for the original formulation. 
We prove that these properties hold for our construction with the following theorems; the first property is that forming the discriminant algebra should commute with base change:

\begin{mainthm}[proven as \cref{base-change}]\label{mainthm-base-change}
 If $R$ is a ring and $A$ is an $R$-algebra of rank $n$, and $R'$ is any $R$-algebra, let $A'$ denote $R'\otimes_R A$, an $R'$-algebra of rank $n$.  
 Then there is a canonical isomorphism $R'\otimes_R \discalg{A/R}\cong 
\discalg{A'/R'}$.
\end{mainthm}

 Deligne's second requirement is an explicit description of the discriminant algebra in the case that $2$ is invertible, and is a generalization of the idea of adjoining to the base ring a square root of the discriminant of the algebra with respect to some basis. 
 
\begin{mainthm}\label{mainthm-2-unit}
 If $2$ is a unit in $R$ and $A$ is an $R$-algebra of rank $n$, then there is an $R$-algebra isomorphism $\discalg{A/R}\simto R\oplus \extpower^nA$, where the latter is given the $R$-algebra structure whose identity element is $(1,0)$, and where the product of two elements $a,b\in\extpower^n A$ is via the discriminant form: $a\cdot b \coloneqq \disc{A/R}(a,b)\in R$.
\end{mainthm}

We will find that it actually makes more sense to include a factor of $1/4$ in the multiplication, so that the product of two elements $a,b\in\extpower^n A$ is $\frac14\disc{A/R}(a,b)$.
(This is analogous to presenting the discriminant algebra of a separable $K$-algebra with discriminant $d$ as $K[x]/(x^2-d/4)$ instead of $K[x]/(x^2-d)$.)
Then \cref{mainthm-2-unit} is actually a consequence of \cref{mainthm-identify-discriminants}; see \cref{discalg-when-2-is-invertible}.

Deligne's third requirement is that in the case of $R$ connected and $A$ \'etale over $R$, the discriminant algebra for $A$ over $R$ should be the one described at the beginning of the introduction:

\begin{mainthm}[proven as \cref{discriminant-algebra-corresponds-to-orientations}]\label{mainthm-etale}
 Let $R$ be a connected ring and $A$ an \'etale $R$-algebra of rank $n$ corresponding to a $\fundgroup{R}$-set $X$. 
 Then $\discalg{A/R}$ is also \'etale, and corresponds to the $2$-element set $\Or(X)$ of orientations of $X$.
\end{mainthm}

Deligne's fourth requirement was that the two descriptions of $\discalg{A/R}$ from \cref{mainthm-etale,mainthm-2-unit} should be compatible when both apply, and we omit the statement here.
In his letter, Deligne sketched a construction of an operation $(R,A)\mapsto \discalgs{Del}{A/R}$ satisfying his four properties.
He first reduced to a similar list of properties that a discriminant algebra operation for quadratic forms should have, and extended the corresponding \'etale and $2$-invertible descriptions across their codimension-2 complement in a universal case. 
This construction only works for odd-rank algebras, so he defined the discriminant algebra of an even-rank $A$ to be that of $R\times A$. 
However, he does not specifically equip $\discalgs{Del}{A/R}$ with the structure of a discriminant-identifying isomorphism $\extpower^2\discalgs{Del}{A/R}\cong \extpower^n A$.

 Earlier, Markus Rost had exhibited in \cite{Rost} a way of producing a natural rank-$2$ algebra $K_{A/R}$ from a rank-$3$ $R$-algebra $A$, but it admits no natural discriminant-identifying isomorphism $\extpower^2 K_{A/R}\cong \extpower^3 A$.
 Rost fixed this defect by the process of Verschiebung or ``shifting'': modifying the multiplication of $K_{A/R}$ to obtain a new quadratic algebra $\discalgs{Rost}{A/R}$ with the desired discriminant form.
 Ottmar Loos later employed this shifting technique to give a discriminant algebra construction in \cite{LoosDiscAlg} for even-rank algebras, extending to the odd-rank case again by defining $\discalgs{Loos}{A/R}\coloneqq\discalgs{Loos}{(R\times A)/R}$ when $A$ has odd rank. 
 
 The main advantages of our discriminant algebra construction $(R,A)\mapsto \discalg{A/R}$ are that it does not appeal to a separate construction for quadratic forms, it does not split into cases based on whether the rank of $A$ is odd or even, and it does not require first handling special cases (such as $2$ being invertible or 
the algebra being \'etale).
 
  Namely, given a ring $R$ and an $R$-algebra $A$ of rank $n$, there is a canonical $R$-algebra homomorphism from the $\S{n}$-invariant tensors of $A^{\otimes n}$ to $R$; 
 we call this the \emph{Ferrand homomorphism} $\ferrand_{A/R}\colon\fixpower{A}{n}{\S{n}}\to R$.  
 Then we use the Ferrand homomorphism to define the discriminant algebra of $A$ over $R$ as the tensor product 
 \[
 \discalg{A/R} \coloneqq \fixpower{A}{n}{\A{n}}\midotimes_{\mathclap{\fixpower{A}{n}{\S{n}}}} R.
 \]
 
In \cref{section-ferrand}, we define the Ferrand homomorphism and discriminant algebra associated to a rank-$n$ algebra.
We prove \cref{mainthm-identify-discriminants} in \cref{section-discalg}, and prove \cref{mainthm-base-change,mainthm-2-unit,mainthm-etale} in \cref{section-discalg-props}.
In between, in \cref{section-understanding} we discuss the action of the Ferrand homomorphism, and \cref{section-examples} exhibits some examples of discriminant algebras.
\Cref{section-functoriality} shows that the construction of the discriminant algebra is functorial with respect to algebra homomorphisms that preserve the characteristic polynomial of every element.

In \cref{section-products}, we analyze the discriminant algebra of a product algebra; 
our main result is that $\discalg{(A\times B)/R}\cong \discalg{A/R}\ast\discalg{B/R}$, where $\ast$ is the commutative monoid structure on the set of isomorphism classes of quadratic algebras, characterized by John Voight in \cite{Voight15}.
We also show that $\discalg{(R\times A)/R}\cong \discalg{A/R}$, so that if we set $\discalg{A/R}\coloneqq R^2$ for $A$ of rank $0$ or $1$, we obtain a monoid homomorphism from the set of isomorphism classes of (locally) constant-rank algebras to the set of isomorphism classes of quadratic algebras, carrying $\times$ to $\ast$.

In forthcoming work, the authors exhibit isomorphisms between this discriminant algebra construction and those given by Rost and Loos.
The isomorphism with Loos's discriminant algebra is quite subtle, and suggests that there may be only one discriminant algebra operation satisfying \cref{mainthm-identify-discriminants,mainthm-base-change,mainthm-2-unit,mainthm-etale} up to isomorphism.
Due to the existence of suitably nice universal cases, the authors have been able to verify this uniqueness in ranks up to $3$, but the general case is still unknown.
  
Both authors would like to thank Lenny Taelman for several helpful discussions throughout this project, as well as Darij Grinberg and the anonymous referee for their comments and corrections.
The second author would also like to thank the Austrian Science Fund FWF for its support through project P25652.

\section{Defining the Ferrand homomorphism and discriminant algebra}\label{section-ferrand}

Given a ring $R$ with algebra $A$ of rank $n$, Daniel Ferrand in \cite{FonctNorme} uses a certain homomorphism $\ferrand_{A/R}\colon \fixpower{A}{n}{\S{n}}\to R$ to construct a functor from $A$-modules to $R$-modules.
We call this homomorphism the \emph{Ferrand homomorphism}. 

In case $K$ is a field and $L$ is a degree-$n$ separable field extension of $K$, then the Ferrand homomorphism $\ferrand_{L/K}$ has a simple description, given the Galois closure $N$ of $L$ over $K$: 
Compile the $n$ homomorphisms $L\to N$ over $K$ into a single homomorphism $L^{\otimes n}\to N$, and restrict it to the subalgebra $\fixpower{L}{n}{\S{n}}$. 
The image in $N$ of any $\S{n}$-invariant tensor will be invariant under the Galois action on $N$, and thus belongs to $K$, giving us the Ferrand homomorphism $\ferrand_{L/K}\colon \fixpower{L}{n}{\S{n}}\to K$.

In this section, we will briefly review Ferrand's abstract definition of the Ferrand homomorphism in \cite{FonctNorme}, and then we will describe how to compute its action on symmetric tensors.

First, we recall the setting in which elements of algebras have norms. 
Thus an algebra over a ring $R$ is just another ring $A$ with a ring homomorphism $R\to A$.

\begin{definition}
Let $R$ be a ring, and let $M$ be an $R$-module. 
We say that $M$ is \emph{locally free (of rank $n$)} if there are elements $r_1,\dots,r_k\in R$, together generating the unit ideal, such that $M_{r_i}$ is free (of rank $n$) as an $R_{r_i}$-module for each $i\in\set{k}$. 
An $R$-algebra $A$ is said to be \emph{of rank $n$} if $A$ is locally free of rank $n$ as an $R$-module.
\end{definition}

\begin{remark}
Equivalently, an $R$-module $M$ is locally free of rank $n$ if and only if it is projective and finitely generated (i.e.\ flat and finitely presented) and the rank of each (necessarily free) $R_{\mathfrak{p}}$-module $M_{\mathfrak{p}}$ is $n$ for each prime ideal $\mathfrak{p}\in M$; see \cite[Thm.\ 2 on p.\ II.141]{BourbakiAlgComm14}.
\end{remark}

The following observation about locally free algebras and modules will be useful throughout this paper:
\begin{lemma}
 Let $R$ be a ring and $A$ an $R$-algebra of rank $n\geq 2$.
 Then the quotient $A/R$ is a locally-free $R$-module of rank $n-1$, and the natural map
 \[\extpower^{n-1}(A/R) \to \extpower^n A\]
 sending $[a_1]\wedge\dots\wedge[a_{n-1}]$ to $1\wedge a_1\wedge\dots\wedge a_{n-1}$ is an isomorphism.
\end{lemma}

In particular, if $A$ is a quadratic (i.e.\ rank-$2$) $R$-algebra then there is a canonical isomorphism $A/R \to \extpower^2 A$ sending the class of $a$ to $1\wedge a$.

\begin{proof}
 For the first claim, that $A/R$ is locally free, it is enough to show that $A/R$ is flat and finitely presented.
 First, the structure morphism $R\to A$ is injective since the rank of $A$ is everywhere positive, and since the same applies after base change to any $R$-algebra, we find that $R\to A$ is the inclusion of a pure submodule.
 The quotient of a flat module by a pure submodule is also flat, so the quotient $A/R$ is a flat $R$-module.
 It is also finitely presented, since $A$ is, so $A/R$ is locally free.
 
 The claims that $A/R$ has rank $n-1$ and that the given map on exterior powers is an isomorphism can be checked locally, so assume $R$ is a local ring.
 Then $A/R$ is a free $R$-module, so let $\{[\theta_1],\dots,[\theta_k]\}$ be an $R$-basis for it.
 Then $A$ is has $R$-basis $\{1, \theta_1,\dots,\theta_k\}$, so since $A$ has rank $n$ we must have $k=n-1$ as desired.
 Finally, note that the indicated homomorphism $\extpower^{n-1}(A/R)\to \extpower^n A$ maps the singleton basis $\{[\theta_1]\wedge\dots\wedge[\theta_{n-1}]\}$ to the basis $\{1\wedge\theta_1\wedge\dots\wedge\theta_{n-1}\}$, so is an isomorphism.
\end{proof}

\begin{definition}
Let $R$ be a ring, and let $A$ be an $R$-algebra of rank $n$.
For each element $a\in A$, multiplication by $a$ is an $R$-module homomorphism $A\to A$, and the $n$th exterior power of this homomorphism is an $R$-module homomorphism $\extpower^n A\to \extpower^n A$. 
Since $\extpower^n A$ is locally free of rank $1$, this endomorphism is equal to multiplication by a unique element of $R$, called the \emph{norm} $\norm_{A/R}(a)$ of $a$.
\end{definition}

What sort of map is $\norm_{A/R}\colon A\to R$?
It is multiplicative, but almost never additive, and thus it is not a ring homomorphism.
It is, however, the base component of a \emph{multiplicative homogeneous degree-$n$ polynomial law}:

\begin{definition}
 Let $R$ be a ring and $M$ and $N$ two $R$-modules.
 A \emph{polynomial law} $p\colon M\to N$ is a collection of functions $p_S\colon S\otimes_R M\to S\otimes_R N$ for each $R$-algebra $S$, such that for every $R$-algebra homomorphism $f\colon S\to S'$ the following square of functions commutes:
 \[\begin{tikzcd}
  S\otimes_R M \arrow{r}{p_S}\arrow{d}[swap]{f\otimes \id_M} & S\otimes_R N \arrow{d}{f\otimes \id_N}\\
  S'\otimes_R M \arrow{r}{p_{S'}} & S'\otimes_R N
 \end{tikzcd}\]
 We say that $p$ is \emph{homogeneous of degree $n$} if for every $R$-algebra $S$, element $s\in S$ and element $m\in S\otimes M$, we have $p_S(sm) = s^n p_S(m)$.
 If $p\colon A\to B$ is a polynomial law between two $R$-algebras $A$ and $B$, we say $p$ is \emph{multiplicative} if each $p_S\colon S\otimes_R A \to S\otimes_R B$ is a multiplicative function.
\end{definition}

For example, if $A$ is an $R$-algebra, then the diagonal function $A\to A^{\otimes n}$ sending $a$ to $a\otimes\dots\otimes a$ extends naturally to a multiplicative homogeneous degree-$n$ polynomial law $A\to A^{\otimes n}$.

A fundamental result of the theory of polynomial laws, developed by Norbert Roby in \cite{PolMaps}, is that for each $R$-module $M$ there is a universal homogeneous degree-$n$ polynomial law $\gamma^n\colon M\to \Gamma^n_R(M)$;
every homogeneous degree-$n$ polynomial law $M\to N$ factors uniquely as $\gamma^n$ followed by an ordinary $R$-module homomorphism $\Gamma^n_R(M)\to N$.
Furthermore, Roby shows in \cite{MultHomMaps} that if $A$ is an $R$-algebra, then $\Gamma^n_R(A)$ is also an $R$-algebra and \emph{multiplicative} degree-$n$ polynomial laws $A\to B$ correspond to $R$-\emph{algebra} homomorphisms $\Gamma^n_R(A)\to B$.
Thus the norm map $\norm_{A/R}\colon A\to R$ corresponds to an $R$-algebra homomorphism $\Gamma^n_R(A)\to R$.

A general presentation for $\Gamma_R(M)$ may be found in \cite[III.1]{PolMaps}, but for our purposes it is sufficient to note that the diagonal polynomial law $A\to A^{\otimes n}$ gives us an $R$-algebra homomorphism $\Gamma^n_R(A)\to A^{\otimes n}$, and that if $A$ is a flat $R$-algebra then this homomorphism restricts to an isomorphism with the subalgebra of $\S{n}$-invariant tensors (see \cite[5.5.2.5 on p.\ 123]{deligne1973cohomologie}):
\[
\Gamma^n_R(A)\simto \fixpower{A}{n}{\S{n}}: \gamma^n(a)\mapsto a\otimes\dots\otimes a.
\]
Thus if $A$ and $B$ are $R$-algebras with $A$ flat, multiplicative homogeneous degree-$n$ polynomial laws $A\to B$ correspond to $R$-algebra homomorphisms $\fixpower{A}{n}{\S{n}}\to B$.


\begin{definition}[{c.f.\ \cite[3.1.2]{FonctNorme}}]\label{definition-ferrand}
Let $R$ be a ring and $A$ an $R$-algebra of rank $n$.
Then the norm polynomial law $\norm_{A/R}\colon A\to R$ corresponds to an $R$-algebra homomorphism 
\begin{align*}
 \ferrand_{A/R}\colon \fixpower{A}{n}{\S{n}}\cong \Gamma^n_R(A) \to R,
\end{align*}
called \emph{the Ferrand homomorphism} (of $A$ over $R$).
It is the unique $R$-algebra homomorphism such that for every $R$-algebra $R'$ and every element $a\in A'\coloneqq R'\otimes_R A$, the composite
\begin{equation}\label{ferrand-base-change}
\begin{tikzcd}[column sep=small]
\fixpower[R']{A'}{n}{\S{n}}\arrow{r}{\sim} & R'\otimes_R \fixpower{A}{n}{\S{n}}\arrow{rrrr}{\id_{R'}\otimes \ferrand_{A/R}} & & & & R'\otimes_R R \arrow{r}{\sim} & R'
\end{tikzcd}
\end{equation}
sends $a\otimes\dots\otimes a$ to $\norm_{A'/R'}(a)\in R'$.
\end{definition}

As an immediate corollary to this definition, we find for each $R$-algebra $R'$ and $A'\coloneqq R'\otimes_R A$ that the composite \eqref{ferrand-base-change} is the Ferrand homomorphism for $A'$ over $R'$.
We say that the Ferrand homomorphism \emph{commutes with base change}.

\begin{definition}
Let $R$ be a ring and let $A$ be an $R$-algebra of rank $n$ with $n\geq 2$. Then the {\em discriminant algebra $\discalg{A/R}$ of $A$ over $R$} is the tensor product of $\fixpower{A}{n}{\S{n}}$-algebras
\[
\discalg{A/R} = \fixpower{A}{n}{\A{n}} \midotimes_{\fixpower{A}{n}{\S{n}}} R 
\]
defined by the inclusion $\fixpower{A}{n}{\S{n}} \into \fixpower{A}{n}{\A{n}}$ and the Ferrand homomorphism $\ferrand_{A/R}\colon \fixpower{A}{n}{\S{n}}\to R$.
\end{definition}

If the base ring $R$ is understood, it may be omitted and the discriminant algebra of $A$ denoted $\discalg{A}$. 
We will adopt similar conventions for the trace and norm maps, the Ferrand homomorphism, and the discriminant bilinear form.

\begin{remark}
 Note that since the Ferrand homomorphism $\ferrand_{A/R}$ is surjective, so is the homomorphism from $\fixpower{A}{n}{\A{n}}\to\discalg{A/R}$.
 Thus $\discalg{A/R}$ can be understood as the quotient of $\fixpower{A}{n}{\A{n}}$ by the ideal generated by elements of the form $x-\ferrand_{A/R}(x)$ for $x\in\fixpower{A}{n}{\S{n}}$.
 If $x$ is a general element of $\fixpower{A}{n}{\A{n}}$, we will often denote its image in $\discalg{A/R}$ by $\dot x$.
\end{remark}

\begin{remark}
 Note that if we replace $\fixpower{A}{n}{\A{n}}$ by $A^{\otimes n}$ in the definition of $\discalg{A/R}$, we obtain what Ferrand denotes by $\smash{\mathbf{P}^{(1,\dots,1)}(A)}$ in \cite[5.2]{FonctNorme} and what Bhargava calls the \emph{$\S{n}$-closure of $A$ over $R$} in \cite{BhargSat}.
 
 Denoting the $\S{n}$-closure by $B$, we thus obtain an $R$-algebra homomorphism $\discalg{A/R}\to B$ by tensoring the inclusion $\fixpower{A}{n}{\A{n}}\into A^{\otimes n}$ along $\ferrand_{A/R}$.
 In case $R\to A$ is a degree-$n$ separable field extension $K\into L$ in characteristic other than $2$, whose normal closure $N$ has Galois group $\S{n}$, then $B\cong N$ and this homomorphism $\discalg{A/R}\to B$ is just the inclusion of the discriminant field of $L$ into $N$. 
\end{remark}

\section{Understanding the Ferrand homomorphism}
\label{section-understanding}

The two defining facts of the Ferrand homomorphisms are that they send elements of the form $a\otimes\dots\otimes a$ to the norm of $a$, and that they commute with base change.
Our primary tool for computing the image of an arbitrary symmetric tensor, then, is to identify it as a term in some tensor power of a single element, and then correspondingly break up the norm of that single element.

For example, if $A$ is any $R$-algebra, and $a\in A$, then we have the following elements of $\fixpower{A}{n}{\S{n}}$:
\begin{align*}
 e_1(a) &\coloneqq  (a\otimes 1\otimes \dots\otimes 1) + (1\otimes a\otimes\dots\otimes 1) + \ldots + (1\otimes\dots\otimes 1\otimes a)
\\
 e_2(a) &\coloneqq  (a\otimes a\otimes 1\otimes\dots\otimes 1) + (a\otimes 1 \otimes a \otimes \dots \otimes 1) + \ldots\\
 &\qquad \ldots + (1\otimes \dots\otimes 1 \otimes a \otimes a)\\
 \ldots & \\
 e_n(a) &\coloneqq  (a\otimes a\otimes\dots\otimes a).
\end{align*}
These are the elementary symmetric polynomials in the $n$ elements $\conjugate{a}{1},\dots,\conjugate{a}{n}\in A^{\otimes n}$, where $\conjugate{a}{i}$ is the element
\begin{equation}\label{conjugate}
  \conjugate{a}{i} = 1\otimes\dots\otimes1\otimes a\otimes 1\otimes\dots\otimes 1
 \end{equation}
with the $a$ in the $i$th tensor factor.

In case $A$ is a rank-$n$ $R$-algebra, we can compute the image of $e_k(a)$ under $\ferrand_{A/R}\colon \fixpower{A}{n}{\S{n}}\to R$ as follows:
First, note that $e_k(a)$ is the coefficient of $\lambda^{n-k}\mu^{k}$ in $(\lambda+\mu a)\otimes\dots\otimes (\lambda+\mu a)$, as an element of $\fixpower[{R[\lambda,\mu]}]{A[\lambda,\mu]}{n}{\S{n}}$.
(Thus we sometimes also write $e_0(a) = 1$.)
The image of this element under $\ferrand_{A[\lambda,\mu]/R[\lambda,\mu]}$ is its norm $f(\lambda,\mu)\in R[\lambda,\mu]$.
Since the norm map is homogeneous of degree $n$, we find that $f(\lambda,\mu)$ is a homogeneous polynomial of degree $n$; denoting by $s_k(a)$ the coefficient of $\lambda^{n-k}\mu^k$ in $f(\lambda,\mu)$, we obtain
\[\ferrand_{A[\lambda,\mu]/R[\lambda,\mu]}\colon \sum_{k=0}^n e_k(a)\lambda^{n-k}\mu^k \mapsto \sum_{k=0}^n s_k(a) \lambda^{n-k}\mu^k.\]
Since $\ferrand_{A[\lambda,\mu]/R[\lambda,\mu]}$ is just $\ferrand_{A/R}\otimes \id_{R[\lambda,\mu]}$, we thus have
\[\ferrand_{A/R}\colon e_k(a) \mapsto s_k(a)\]
for each $k\in\set[0]{n}$.

What are these quantities $s_k(a)\in R$?
The answer comes by setting $\mu=-1$, so that
\[\sum_{k=0}^n s_k(a)\lambda^{n-k}(-1)^k = \norm_{A[\lambda]/R[\lambda]}(\lambda-a),\]
the characteristic polynomial of $a$.
We can thus read off $s_k(a)$ as $(-1)^k$ times the coefficient of $\lambda^{n-k}$ in the characteristic polynomial of $a$.
In particular, $s_n(a)$ is the norm $\norm_{A/R}(a)$ of $a$, and $s_1(a)$ is its trace $\trace_{A/R}(a)$.

We summarize the above discussion in the following lemma:

\begin{lemma}\label{ek-sent-to-sk}
 Let $R$ be a ring and $A$ an $R$-algebra with an element $a\in A$, and let $k\in\set[0]{n}$.
 Denote by $e_k(a)$ the $k$th elementary symmetric polynomial in the $n$ elements $\conjugate{a}{1},\dots,\conjugate{a}{n}\in A^{\otimes n}$ defined in \eqref{conjugate}.
 Now suppose $A$ is a rank-$n$ $R$-algebra, and let $s_k(a)$ be $(-1)^k$ times the coefficient of $\lambda^{n-k}$ in the characteristic polynomial of $a$.
 Then $\ferrand_{A/R}\bigl(e_k(a)\bigr) = s_k(a)$.
\end{lemma}

We can similarly compute the image under $\ferrand_{A/R}$ of other elements of $\fixpower{A}{n}{\S{n}}$.

\begin{definition}[{c.f.\ \cite[2.2.3.1]{FonctNorme}}]\label{definition-gamma-alpha}
Let $a = (a_i)_{i\in I} \in A^I$ be a tuple of elements of $A$, and let $\alpha\subset I^n$.
We define an element $\gamma^{\alpha}(a)\in A^{\otimes n}$ by
\[
\gamma^\alpha(a) \coloneqq  \sum_{(i_1,\dots,i_n)\in \alpha}a_{i_1}\otimes\dots \otimes a_{i_n}.
\]
If $G$ is a subgroup of $\S{n}$ and $\alpha$ is invariant under the action of $G$ on $I^n$, then $\gamma^\alpha(a)\in\fixpower{A}{n}{G}$.
For example, if $\alpha$ is the $\S{n}$-orbit of $(i_1,\dots,i_n)\in I^n$, then $\gamma^\alpha(a)$ is the coefficient of $\lambda_{i_1}\dots\lambda_{i_n}$ in $(\sum_{i\in I}\lambda_ia_i)\otimes\dots\otimes(\sum_{i\in I}\lambda_ia_i)$.
In this case $\ferrand_{A/R}(\gamma^\alpha(a))$ is the coefficient of $\lambda_1\dots\lambda_n$ in the norm of $\sum_{i\in I}\lambda_i a_i$, as an element of the rank-$n$ $R[\lambda_i:i\in I]$-algebra $A[\lambda_i:i\in I]$.
The case of $\ferrand_{A/R}(e_k(a)) = s_k(a)$ is recovered by taking $I=\{1,2\}$ and $(a_1,a_2) = (1,a)$.

In general, we will write the set of $G$-orbits of $I^n$ as $I^n/G$, and write $\alpha\in I^n/G$ to mean that $\alpha$ is one such $G$-orbit.
\end{definition}

We will often consider the case that the tuple $I\to A$ is an ordinary $n$-tuple $a = (a_1,\dots,a_n)\in A^n$, and that the subset $\alpha\subset \set{n}^n = \Map(\set{n},\set{n})$ is a subgroup $G\subset\S{n}$, or a coset $\sigma G$ of $G$.
(The most common cases are $G=\S{n}$ or $\A{n}$, or the coset of odd permutations $\Abar{n}$.)
The quantity $\gamma^{\sigma G}(a_1,\dots,a_n)$ is linear in each $a_i$, since each $a_i$ appears exactly once per term as one of the tensor factors.

\begin{example}\label{ferrand-example}
 Consider the case $R=\Z$ and $A=R[x]/(x^2+x+2)$, a quadratic $R$-algebra.
 If we set $a_1 = -2x+1$ and $a_2=3x+2$, then the element $a_1\otimes a_2+a_2\otimes a_1$ is an $\S{2}$-invariant element of $A^{\otimes 2}$---what is its image under the Ferrand homomorphism $\ferrand_{A/R}\colon\fixpower{A}{2}{\S{2}}\to R$? 
 We will calculate the answer three different ways to demonstrate the various methods we will use throughout the paper.
 
 First, and simplest, note that since $x\in A$ satisfies both the defining equation $x^2 + x + 2 = 0$ as well as its characteristic polynomial $x^2 - \trace_A(x) + \norm_A(x) = 0$, we find that these two equations are equal because $1$ and $x$ are $R$-linearly independent elements of $A$.
 Therefore $\trace_A(x) = -1$ and $\norm_A(x)=2$.
 We can then compute the image of $a_1\otimes a_2 + a_2\otimes a_1$ by expanding it in terms of these quantities:
 \begin{align*}
  a_1\otimes a_2 + a_2\otimes a_1 &= (-2x+1)\otimes(3x+2) + (3x+2)\otimes(-2x+1)\\
  &= -6(x\otimes x) -4(x\otimes 1) + 3(1\otimes x) + 2(1\otimes 1)\\
  &\quad -6(x\otimes x) +3(x\otimes 1) - 4(1\otimes x) + 2(1\otimes 1)\\
  &= -12(x\otimes x) - 1(x\otimes 1 + 1\otimes x) + 4\\
  &= -12\cdot e_2(x) - 1\cdot e_1(x) + 4\\
  &\mapsto -12\cdot s_2(x) - 1\cdot s_1(x) + 4\\
  &= -12\cdot\norm_{A}(x) - 1\cdot\trace_{A}(x) + 4\\
  &= -12\cdot 2 - 1\cdot -1 + 4 = -19.
 \end{align*}

We can also determine the image of $a_1\otimes a_2 + a_2\otimes a_1$ knowing only how multiplying by $a_1$ and $a_2$ act with respect to an $R$-basis of $A$.
Choosing the $R$-basis $\{1,x\}$, we find that they act by the matrices
\[a_1:\begin{pmatrix}
 1 & 4\\
 -2 & 3
\end{pmatrix}\qquad
 a_2:\begin{pmatrix}
 2 & -6\\
 3 & -1
\end{pmatrix}.
\]
(For example, $a_2\cdot x = (3x+2)(x) = 3x^2 + 2x = 3(-x-2) + 2x = -6-x$, whose coefficients are found in the second column of the $a_2$ matrix.)
Now $a_1\otimes a_2+a_2\otimes a_1$ is the quantity denoted $\gamma^{\S{2}}(a_1,a_2)$, which is the coefficient of $\lambda_1\lambda_2$ in $(\lambda_1a_1 + \lambda_2a_2)\otimes (\lambda_1 a_1+ \lambda_2 a_2)$.
Thus its image under the Ferrand homomorphism is the coefficient of $\lambda_1\lambda_2$ in the norm of $\lambda_1 a_1+\lambda_2 a_2$ (as an element of the quadratic $R[\lambda_1,\lambda_2]$-algebra $A[\lambda_1,\lambda_2]$).
This element acts by the matrix
\[\lambda_1a_1 + \lambda_2 a_2 : \begin{pmatrix}
 \lambda_1+2\lambda_2 & 4\lambda_1-6\lambda_2\\
 -2\lambda_1+3\lambda_2 & 3\lambda_1-\lambda_2
\end{pmatrix}.\]
Ignoring all but the $\lambda_1\lambda_2$ terms, this determinant is
\[(\lambda_1)(-\lambda_2) + (2\lambda_2)(3\lambda_1) - (4\lambda_1)(3\lambda_2) - (-6\lambda_2)(-2\lambda_1) = -19\lambda_1\lambda_2,\]
so $\ferrand_{A/R}(a_1\otimes a_2 + a_2\otimes a_1) = -19$, as before.

Finally, note that we can write $a_1\otimes a_2 + a_2\otimes a_1$ as a polynomial in symmetric tensors of the form $e_k(f)$:
\begin{align*}
 a_1\otimes a_2 + a_2\otimes a_1 &= (a_1\otimes 1 + 1\otimes a_1)(1\otimes a_2 + a_2\otimes 1)\\
 &\qquad - (a_1a_2\otimes 1 + 1\otimes a_1a_2)\\
 &= e_1(a_1)e_1(a_2) - e_1(a_1a_2)\\
 &\mapsto s_1(a_1)s_1(a_2) - s_1(a_1a_2)\\
 &= \trace_A(a_1)\trace_A(a_2) - \trace_A(a_1a_2).
\end{align*}
 We can read off $\trace_A(a_1)$ and $\trace_B(a_2)$ as the traces of the above matrices, namely $1+3 = 4$ and $2-1 = 1$.
 The trace of $a_1a_2$ can by computed similarly as the trace of the product matrix:
 \[a_1a_2:\begin{pmatrix}
 1 & 4\\
 -2 & 3
\end{pmatrix}\begin{pmatrix}
 2 & -6\\
 3 & -1
\end{pmatrix} = \begin{pmatrix}
 14 & -10\\
 5 & 9
\end{pmatrix},\]
so $\trace_A(a_1a_2) = 14+9 = 23$.
Therefore $\ferrand_A(a_1\otimes a_2+a_2\otimes a_1) = (4)(1) - (23) = -19$.
\end{example}

The latter two methods for computing $\ferrand_{A/R}$ of an $\S{n}$-invariant tensor---expanding the tensor as a linear combination of the $\gamma^\alpha(a)$ or a polynomial in the $e_k(a)$---will be used throughout this paper.
The remainder of this section consists of technically useful results on sets of module and algebra generators for the $G$-invariants of tensor powers, which, in effect, say that the methods of \cref{ferrand-example} will always be sufficient to calculate the image of an $\S{n}$-invariant tensor under the Ferrand homomorphism.

%

\begin{lemma}\label{generators-for-G-invariants}
Let $R$ be a ring and $M$ a projective $R$-module.
Let $\theta = (\theta_i)_{i\in I} \in M^I$ be a generating family (resp.\ $R$-basis) for $M$. 
Then for each natural number $d$ and subgroup $G\subset\S{d}$, the family $\{\gamma^\alpha(\theta): \alpha\in I^d/G\}$ is a generating family (resp.\ $R$-basis) for $\fixpower{M}{d}{G}$.
\end{lemma}

Recall that by $\alpha\in I^d/G$ we mean that $\alpha$ is an orbit of $I^d$ under the action of $G$, so that the notation $\gamma^\alpha(\theta)$ makes sense.

\begin{proof}
 First consider the case in which $\theta$ is an $R$-basis for $M$.  Then every tensor in $M^{\otimes d}$ can be uniquely represented as an $R$-linear combination of pure tensors of the form $\theta_{i_1}\otimes\dots\otimes \theta_{i_d}$.  
Such a linear combination is $G$-invariant precisely if the coefficients are constant across $G$-orbits of the pure tensors, i.e.\ if the tensor is an $R$-linear combination of the $\gamma^\alpha(\theta)$ as $\alpha$ 
ranges over $I^d/G$.  
Hence the $\gamma^\alpha(\theta)$ generate $\fixpower{M}{d}{G}$.  
Furthermore, since no pure tensor is a term in more than one of the $\gamma^\alpha(\theta)$, they are also $R$-linearly independent, and hence form an $R$-basis for $\fixpower{M}{d}{G}$.
 
 Second, suppose that $M$ is merely projective and that $\{\theta_i:i\in I\}$ merely generates $M$ as an $R$-module.  
 Then let $R^{(I)}$ be the free $R$-module with basis $e=(e_i)_{i\in I}$, and $R^{(I)}\onto M$ be the surjection sending $e_i$ 
to $\theta_i$.  
Since $M$ is projective, this surjection has a right inverse, which is a property preserved by any functor.  
Hence after applying the functor $\fixpower{(\cdot)}{d}{G}$, we obtain another surjection $\fixpower{(R^{\smash{(I)}})}{d}{G}\onto\fixpower{M}{d}{G}$.  
Since $\{\gamma^\alpha(e): \alpha\in I^d/G\}$ generates $\fixpower{(R^{\smash{(I)}})}{d}{G}$ by the above, its image $\{\gamma^\alpha(\theta): \alpha\in I^d/G\}$ generates $\fixpower{M}{d}{G}$ as well.
\end{proof}

Thus given a set of $R$-module generators $\theta=(\theta_i)_{i\in I}$ of a rank-$n$ $R$-algebra $A$, we find that every element of $\fixpower{A}{n}{\S{n}}$ (resp.\ $\fixpower{A}{n}{\A{n}}$) can be written as an $R$-linear combination of the $\gamma^\alpha(\theta)$, as $\alpha$ ranges over all $\S{n}$-orbits (resp.\ $\A{n}$-orbits) of $I^n$.
This fact has a number of consequences that we explore in the remainder of this section.

\begin{proposition}\label{fixpower-commutes-with-base-change}
Let $R$ be a ring, let $M$ be a locally free $R$-module, let $d$ be a natural number, and let $G$ be a subgroup of $\S{d}$. 
Let $R^\prime$ be any $R$-algebra, and denote by $M'$ the locally free $R'$-module $R'\otimes_R M$. 
Then the natural $R'$-module homomorphism $R'\otimes_R \fixpower{M}{d}{G}\to \fixpower[R']{M'}{d}{G}$ is an isomorphism.
\end{proposition}

\begin{proof}
 Since $\fixpower{M}{d}{G}$ can be written as an equalizer of finitely many maps $M^{\otimes d}\to M^{\otimes d}$, the functor $\fixpower{(\cdot)}{d}{G}$ commutes with localization (see, for example, \cite[Prop.\ A7.1.3]{Katz-Mazur}). 
 We therefore only need to check that the given homomorphism is an isomorphism in the case that $M$ is free.  
 Say $M$ has $R$-basis $\theta=(\theta_i)_{i\in I}$, so that $M'$ has a free $R'$-basis $\theta'=(1\otimes\theta_i)_{i\in I}$. 
Then the canonical map $R'\otimes_R \fixpower{M}{d}{G} \to \fixpower[R']{M'}{d}{G}$ carries the $R'$-basis $\{1\otimes \gamma^\alpha(\theta):\alpha\in I^d/G\}$ to the $R'$-basis $\{\gamma^\alpha(\theta'):\alpha\in I^d/G\}$, so is an isomorphism.
\end{proof}

Next we show that elements of the form $e_k(a)$ generate $\fixpower{A}{n}{\S{n}}$ as an algebra.
The proof is algorithmic in nature, and shows us how to write any element of the form $\gamma^{\alpha}(a)$ as a polynomial in the $e_k(a)$.

\begin{lemma}\label{ek-generators}
 Let $R$ be a ring and $A$ a rank-$n$ $R$-algebra.
 Let $\Omega\subset A$ be a set of elements of $A$ whose powers together generate $A$ as an $R$-module.
 Then the ring $\fixpower{A}{n}{\S{n}}$ is generated as an $R$-algebra by  $\{e_k(a): k\in\set{n}\text{ and }a\in\Omega\}$.
\end{lemma}

 In particular, the Ferrand homomorphism $\ferrand_{A/R}\colon\fixpower{A}{n}{\S{n}}\to R$ is the unique $R$-algebra homomorphism sending $e_k(a)\mapsto s_k(a)$ for all $k\in\set{n}$ and $a\in \Omega$.
(Note: in case $\Omega$ is the set of primitive monomials of a set of algebra generators for $A$, \cref{ek-generators} is a corollary of \cite[Theorem 4.10]{Vac08} by an argument similar to the proof of \cref{generators-for-G-invariants}.)

\begin{proof}
 First note that for each $m\in\N$ and $a\in \Omega$, we can express each $e_k(a^m)$ as a polynomial in $\{e_j(a):j\in\set{n}\}$ by the fundamental theorem of symmetric polynomials.
 Therefore by adding to $\Omega$ all powers of its elements, we can assume without loss of generality that $\Omega$ is a set of $R$-module generators for $A$ that contains $1$.
 Therefore the set $\{\gamma^\alpha(\Omega):\alpha\in\Omega^n/\S{n}\}$ generates $\fixpower{A}{n}{\S{n}}$ as an $R$-module, where we regard $\Omega$ as a family of elements of $A$ indexed by itself.
 
 Next we show by induction on $k$ that if $\alpha\in \Omega^n/\S{n}$ is the $\S{n}$-orbit of any tuple with at most $k$ entries not equal to $1$, then $\gamma^{\alpha}(\Omega)$ is in the subalgebra of $\fixpower{A}{n}{\S{n}}$ generated by $\{e_j(a):a\in\Omega\text{ and }j\in\set{k}\}$.
 The case $k=n$ then implies that the subalgebra generated by all the $e_k(a)$ contains a set of $R$-module generators for $\fixpower{A}{n}{\S{n}}$, and thus is all of $\fixpower{A}{n}{\S{n}}$.
 
 The base case $k=0$ is clear: if $\alpha = \{(1,\dots,1)\}$, then $\gamma^\alpha(\Omega) = 1$ is in every subalgebra of $\fixpower{A}{n}{\S{n}}$.
 
 Now assume that the hypothesis holds for all $j<k$, and consider an element of the form $\gamma^{\alpha}(\Omega)$ where $\alpha$ is the $\S{n}$-orbit of a tuple in $\Omega^n$ with at most $k$ components unequal to $1$.
 Given $a\in\Omega$, let $\alpha(a)$ be the number of times $a$ appears in any single tuple in $\alpha$; thus $\alpha(1)\geq n-k$ and $\sum_{a\in\Omega\setminus\{1\}}\alpha(a) \leq k$.
 Construct the element 
 \[p=\prod_{a\in\Omega\setminus\{1\}}e_{\alpha(a)}(a)\in \fixpower{A}{n}{\S{n}};\]
  since $\alpha(a)=0$ (and thus $e_{\alpha(a)}(a) = 1$) for all but finitely many $a\in\Omega$, this product is well-defined.
 Rewrite the product $p$ as follows: first simplify the product and collect terms related by permutations to write $p$ as a sum of terms of the form $\gamma^\beta(A)$, one of which is the original $\gamma^\alpha(\Omega)$.
 For the remainder of the terms, choose an expression of each tensor factor as an $R$-linear combination of elements of $\Omega$, taking care to choose the same expression each time the same element of $A$ appears and to leave each tensor factor $1$ alone. 
 After collecting like terms again, we have written $p$ as $\gamma^\alpha(\Omega)$ plus an $R$-linear combination of terms of the form $\gamma^\beta(\Omega)$.
 These other $\beta$ all have the property that $\beta(1)>n-k$, so by the induction hypothesis $\gamma^\alpha(a)-p$ is in the algebra generated by $\{e_j(a):a\in\Omega\text{ and }j<k\}$.
 Therefore $\gamma^\alpha(a)$ is in the subalgebra generated by $\{e_j(a):a\in\Omega\text{ and }j\leq k\}$, as desired.
\end{proof}

For example, given a tuple $a = (a_1,\dots,a_n)$ of elements of an algebra $A$, we have the following elements of $A^{\otimes n}$:
\begin{align*}
 \gamma^{\A{n}}(a) &= \sum_{\sigma\in\S{n}\text{ even}} a_{\sigma(1)}\otimes\dots\otimes a_{\sigma(n)}\text{ and}\\
 \gamma^{\Abar{n}}(a) &= \sum_{\sigma\in\S{n}\text{ odd}} a_{\sigma(1)}\otimes\dots\otimes a_{\sigma(n)},
 \intertext{where $\Abar{n}$ is the nontrivial coset of $\A{n}$, and their sum}
  \gamma^{\S{n}}(a) &= \sum_{\sigma\in\S{n}} a_{\sigma(1)}\otimes\dots\otimes a_{\sigma(n)}.
\end{align*}

%


Finally, we write the product of two elements of the form $\gamma^{G}(a_1,\ldots, a_n)$ in terms of more elements of this form.

\begin{lemma}\label{multiplication-on-Delta}
Let $R$ be a ring and $A$ an $R$-algebra. 
Let $n$ be a natural number and $G$ a subgroup of $\S{n}$. 
Then for each pair of tuples $(a_1,\dots,a_n)$ and $(b_1,\dots, b_n)$ in $A^n$, the equation
\[
\gamma^{G}(a)\gamma^{G}(b) = \sum_{\sigma\in G} \gamma^{G}(a_1 b_{\sigma(1)},\dots, a_n b_{\sigma(n)}).
\]
holds in $\fixpower{A}{n}{G}$.
\end{lemma}
If we denote the tuple  $(b_{\sigma(1)},\dots, b_{\sigma(n)})$ by $b_\sigma\in A^n$, and give $A^n$ the product algebra structure, then the product $ab_\sigma$ is $(a_1b_{\sigma(1)},\dots,a_nb_{\sigma(n)})$, and we may write the above identity as
\[
\gamma^{G}(a)\gamma^{G}(b) = \sum_{\sigma\in G} \gamma^G(ab_\sigma).
\]
\begin{proof}
By definition $\gamma^{G}(a)$ equals $\sum_{\tau \in G} \prod_{i=1}^n \conjugate{a_{\tau(i)}}{i}$---recalling the $\conjugate{\vphantom{a}}{i}$ notation for the $i$th embedding $A\to A^{\otimes n}$---so we have
\[
\gamma^{G}(a) \gamma^{G}(b) = \left(\sum_{\tau\in G} \prod_{i=1}^n \conjugate{a_{\tau (i)}}{i}\right)\left(\sum_{\sigma\in G} \prod_{i=1}^d \conjugate{b_{\sigma (i)}}{i}\right).
\]
Expanding the product we get
\[
\sum_{\tau\in G}\sum_{\sigma\in G}\prod_{i=1}^{n} a_{\tau(i)}^{(i)}b_{\sigma(i)}^{(i)} = \sum_{\tau\in G}\sum_{\sigma\in G}\prod_{i=1}^{n}\conjugate{a_{\tau (i)}}{i}\conjugate{b_{\sigma\tau(i)}}{i}
\]
since for fixed $\tau$ the composite $\sigma\tau$ runs over $G$ as $\sigma$ does. 
Now $\conjugate{a_{\tau (i)}}{i}\conjugate{b_{\sigma\tau(i)}}{i} = \conjugate{(ab_{\sigma})_{\tau(i)}}{i}$, so by reversing the order of summation we find
\[
\gamma^{G}(a) \gamma^{G}(b) = \sum_{\sigma\in G}\sum_{\tau\in G}\prod_{i=1}^n \conjugate{(ab_{\sigma})_{\tau(i)}}{i} = \sum_{\sigma\in 
G}\gamma^G(ab_{\sigma})
\]
as desired.
\end{proof}

\section{Proof of \cref*{mainthm-identify-discriminants}}
\label{section-discalg}

Recall that given a ring $R$ with a rank-$n\geq 2$ algebra $A$, the discriminant algebra of $A$ (over $R$) is defined to be the tensor product
\[\discalg{A} = \fixpower{A}{n}{\A{n}}\ \midotimes_{\mathclap{\fixpower{A}{n}{\S{n}}}}\ R,\]
using the two ring homomorphisms $\fixpower{A}{n}{\S{n}}\into\fixpower{A}{n}{\A{n}}$ and the Ferrand homomorphism $\ferrand_{A}\colon \fixpower{A}{n}{\S{n}}\to R$.
Equivalently, since the Ferrand homomorphism is surjective, we could present $\discalg{A}$ as the quotient of $\fixpower{A}{n}{\A{n}}$ by the ideal generated by elements of the form $\{x-\ferrand_{A}(x):x\in\fixpower{A}{n}{\S{n}}\}$.

In this section, we prove \cref{mainthm-identify-discriminants} in the following form:

\begin{theorem}\label{exact-sequence}
Let $R$ be a ring, and let $A$ be an $R$-algebra of rank $n$, with $n\geq 2$.
\begin{enumerate}[label=(\alph*), ref=\cref{exact-sequence}(\alph*)]
 \item \label{exact-sequence-description}
 Its discriminant algebra $\discalg{A}$ fits naturally into a short exact sequence of $R$-modules
 \[
 \begin{tikzcd}
 0\arrow{r} & R \arrow{r} & \discalg{A} \arrow{r} & \extpower^n A \arrow{r} & 0.
 \end{tikzcd}
 \]
 In particular, $\discalg{A}$ is an $R$-algebra of rank $2$.
\item \label{discriminants-identified}
 The resulting composition of isomorphisms $\extpower^n A \cong \discalg{A}/(R\cdot 1) \cong \extpower^2\discalg{A}$ 
 identifies the discriminant bilinear forms of $A$ and $\discalg{A}$.
\end{enumerate}
\end{theorem}

In fact, if we denote for each $(a_1,\dots,a_n)\in A^n$ the image of the $\A{n}$-invariant tensor
\[\gamma^{\A{n}}(a_1,\dots,a_n) \coloneqq  \sum_{\sigma\in \A{n}}a_{\sigma(1)}\otimes\dots\otimes a_{\sigma(n)},\]
 under the surjection $\fixpower{A}{n}{\A{n}}\onto \discalg{A}$ by $\dot\gamma^{\A{n}}(a_1,\dots,a_n)$, then we can describe the $R$-module homomorphism $\discalg{A}\to\extpower^n A$ in the exact sequence as the unique one sending $1$ to $0$ and each $\dot\gamma^{\A{n}}(a_1,\dots,a_n)$ to $a_1\wedge\dots\wedge a_n$.
The isomorphism $\extpower^n A\cong \extpower^2\discalg{A}$ then sends $a_1\wedge\dots\wedge a_n$ to $1\wedge \dot\gamma^{\A{n}}(a_1,\dots,a_n)$.

The outline of the proof of \cref{exact-sequence-description} is as follows.
We show first that the $\fixpower{A}{n}{\S{n}}$-linear map $A^{\otimes n}\to \fixpower{A}{n}{\A{n}}$ sending $a_1\otimes\dots\otimes a_n$ to $\gamma^{\A{n}}(a_1,\dots,a_n)$ descends to an isomorphism of their quotient modules $\extpower^n A\cong \fixpower{A}{n}{\A{n}}/\fixpower{A}{n}{\S{n}}$.
This gives us a short exact sequence:
\[
 \begin{tikzcd}
 0\arrow{r} & \fixpower{A}{n}{\S{n}} \arrow{r} & \fixpower{A}{n}{\A{n}} \arrow{r} & \extpower^n A \arrow{r} & 0.
 \end{tikzcd}
\]
Then we show that base changing this exact sequence along the Ferrand homomorphism $\ferrand_{A/R}\colon\fixpower{A}{n}{\S{n}}\to R$ gives us the desired exact sequence of $R$-modules in \cref{exact-sequence-description}.
The following lemma gives the technical results needed to make this argument go through.

\begin{lemma}\label{rank-n-props}
 Let $R$ be a ring and $A$ an $R$-algebra of rank $n\geq 2$.
 \begin{enumerate}[ref=(\arabic*)]
  \item \label{fixpower-module-structure}
  The $R$-linear map $A^{\otimes n}\to \fixpower{A}{n}{\A{n}}$ sending $a_1\otimes\dots\otimes a_n$ to $\gamma^{\A{n}}(a_1,\dots,a_n)$ descends to an isomorphism $\extpower^n A \cong \fixpower{A}{n}{\A{n}}/\fixpower{A}{n}{\S{n}}$.
  In particular, $\extpower^n A$ inherits an $\fixpower{A}{n}{\S{n}}$-module structure from $A^{\otimes n}$.
  \item \label{ferrand-via-extpower}
  The resulting homomorphism from $\fixpower{A}{n}{\S{n}}$ to $\End_R(\extpower^n A)\cong R$ is the Ferrand homomorphism $\ferrand_{A/R}$.
 \end{enumerate}
\end{lemma}

\begin{proof}  
  \cref{fixpower-module-structure} The $R$-module homomorphism $A^{\otimes n}\to \fixpower{A}{n}{\A{n}}$ sending $a_1\otimes\dots\otimes a_n$ to $\gamma^{\A{n}}(a_1,\dots,a_n)$ is well-defined since $\gamma^{\A{n}}(a_1,\dots,a_n)$ is $R$-linear in each $a_i$.
  It is also easily seen to be $\fixpower{A}{n}{\S{n}}$-linear.
  Therefore we obtain an $\fixpower{A}{n}{\S{n}}$-module homomorphism $A^{\otimes n}\to \fixpower{A}{n}{\A{n}}/\fixpower{A}{n}{\S{n}}$ sending $a_1\otimes\dots\otimes a_n$ to the class of $\gamma^{\A{n}}(a_1,\dots,a_n)$.
  But if any $a_i=a_j$ with $i\neq j$, we find that $\gamma^{\A{n}}(a_1,\dots,a_n)$ is $\S{n}$-invariant and thus equal to zero in the quotient.
  Thus we obtain a well-defined $R$-module homomorphism
  \[\extpower^n A \to \fixpower{A}{n}{\A{n}}/\fixpower{A}{n}{\S{n}}\]
  sending $a_1\wedge\dots\wedge a_n$ to the class of $\gamma^{\A{n}}(a_1,\dots,a_n)$.
  
  Since both sides commute with localization (see \cref{fixpower-commutes-with-base-change} for the right-hand side), it suffices to check that this homomorphism is an isomorphism in the case that $A$ is free as an $R$-module, say with basis $(\theta_1,\dots,\theta_n)$.
  Then $\extpower^n A$ is free with singleton basis $\theta_1\wedge\dots\wedge \theta_n$.
  Applying \cref{generators-for-G-invariants}, we also find that
  \[
  \fixpower{A}{n}{\A{n}} = \bigoplus_{\mathclap{\alpha\in \set{n}^n/\A{n}}}R\cdot{\gamma^\alpha(\theta)}
  \qquad\text{ and }\qquad
  \fixpower{A}{n}{\S{n}} = \bigoplus_{\mathclap{\alpha\in \set{n}^n/\S{n}}}R\cdot{\gamma^\alpha(\theta)}.
  \]
  Now any $\A{n}$-orbit of a tuple with a repeated index is actually its $\S{n}$-orbit, so these bases are identical except that one uses $\gamma^{\A{n}}(\theta)$ and $\gamma^{\Abar{n}}(\theta)$, and the other their sum $\gamma^{\S{n}}(\theta)$.
  Thus the quotient is
  \[\fixpower{A}{n}{\A{n}}/\fixpower{A}{n}{\S{n}} 
  = \frac{R\cdot{\gamma^{\A{n}}(\theta)}\oplus R\cdot{\gamma^{\Abar{n}}(\theta)}}{R\cdot({\gamma^{\A{n}}(\theta)+\gamma^{\Abar{n}}(\theta)})},\]  
  for which we may take either the class of $\gamma^{\A{n}}(\theta)$ or $\gamma^{\Abar{n}}(\theta)$ as a singleton $R$-module basis.
  Then under the homomorphism $\extpower^n A\to \fixpower{A}{n}{\A{n}}/\fixpower{A}{n}{\S{n}}$, the free $R$-module generator $\theta_1\wedge\dots\wedge\theta_n$ for $\extpower^n A$ is transformed into the class of $\gamma^{\A{n}}(\theta_1,\dots,\theta_n)$, which we now know to also be a free $R$-module generator for $\fixpower{A}{n}{\A{n}}/\fixpower{A}{n}{\S{n}}$.
  
  \cref{ferrand-via-extpower} We check the defining property of the Ferrand homomorphism.
  First, let $a\in A$; we will see which endomorphism of $\extpower^n A$ the element $a\otimes\dots\otimes a\in\fixpower{A}{n}{\S{n}}$ gives rise to.
  The $\fixpower{A}{n}{\S{n}}$-module structure on $\extpower^n A$ is the one inherited from the $\fixpower{A}{n}{\S{n}}$-algebra $A^{\otimes n}$, so the endomorphism of $\extpower^n A$ corresponding to $a\otimes\dots\otimes a$ is
  \[b_1\wedge\dots\wedge b_n \mapsto (ab_1)\wedge\dots\wedge(ab_n).\]
  But by definition, this endomorphism is multiplication by $\norm_{A/R}(a)\in R$, as desired.
  
  Now we check that for each $R$-algebra $R'$, if we set $A'=R'\otimes_R A$ then the base-changed homomorphism 
  \[\fixpower[R']{A'}{n}{\S{n}}\cong R'\otimes_R \fixpower{A}{n}{\S{n}} \to R'\otimes_R \End_R(\extpower^n A) \cong \End_{R'}(\extpower^n A')\]
  has the same property, sending each $a\otimes\dots\otimes a$ to multiplication by $\norm_{A'/R'}(a)\in R'$.
  But this follows because the base-changed homomorphism is again the one expressing the fact that $\extpower^n A'$ inherits an $\fixpower[R']{A'}{n}{\S{n}}$-module structure from $A'^{\otimes_{R'} n}$.
\end{proof}

\begin{proof}[Proof of \cref{exact-sequence-description}]
 By \cref{rank-n-props}\cref{fixpower-module-structure}, we obtain a short exact sequence of $\fixpower{A}{n}{\S{n}}$-modules
 \[0\to \fixpower{A}{n}{\S{n}}\to \fixpower{A}{n}{\A{n}}\to \extpower^n A\to 0\]
 where the right-hand map sends $\gamma^{\A{n}}(a_1,\dots,a_n)\mapsto a_1\wedge\dots\wedge a_n$.
 Tensoring along the Ferrand homomorphism $\ferrand_{A/R}\colon \fixpower{A}{n}{\S{n}}\to R$ is equivalent to quotienting by its kernel $I = (x-\ferrand_{A}(x):x\in\fixpower{A}{n}{\S{n}})$, obtaining an \emph{a priori} merely right-exact sequence
 \[
 \fixpower{A}{n}{\S{n}}/I \to \fixpower{A}{n}{\A{n}}/\bigl(I\cdot \fixpower{A}{n}{\A{n}}\bigr) \to \extpower^n A / \bigl(I\cdot\extpower^n A\bigr) \to 0.
 \]
 From left to right, these modules are $R$, $\discalg{A}$, and $\extpower^n A$, in the latter case because $I$, being the kernel of the Ferrand homomorphism, is also the annihilator of the $\fixpower{A}{n}{\S{n}}$-module $\extpower^n A$ by \cref{rank-n-props}\cref{ferrand-via-extpower}.
  Thus we have a right exact sequence of $R$-modules
 \[R\to \discalg{A} \to \extpower^n A\to 0,\]
 where the left-hand map sends $1\mapsto 1$ and the right-hand map sends each $\dot\gamma^{\A{n}}(a_1,\dots,a_n)$ to $a_1\wedge\dots\wedge a_n$.
 But then $R\to\discalg{A}$ must be injective: if $r\mapsto 0$, then $r$ acts by zero on $\discalg{A}$ and hence on its quotient $\extpower^n A$.
 But then $r$ must be zero, since $\extpower^n A$ is locally free of rank $1$ and thus a faithful $R$-module.

 (Another way to see that $R\to \discalg{A}$ is injective is to note that $I$ is the annihilator of the $\fixpower{A}{n}{\S{n}}$-module $\fixpower{A}{n}{\A{n}}/\fixpower{A}{n}{\S{n}}$.
 It is thus the largest ideal of $\fixpower{A}{n}{\S{n}}$ whose product with $\fixpower{A}{n}{\A{n}}$ is contained in $\fixpower{A}{n}{\S{n}}$; therefore $I\cdot\fixpower{A}{n}{\A{n}} = I$.
 Thus the map $R\to \discalg{A}$ is the map
 \[
 \fixpower{A}{n}{\S{n}}/I \to \fixpower{A}{n}{\A{n}}/\bigl(I\cdot \fixpower{A}{n}{\A{n}}\bigr) = \fixpower{A}{n}{\A{n}}/I,
 \]
 which is manifestly injective.)
\end{proof}


The remainder of the section is devoted to proving \cref{discriminants-identified}. 
Our first result relates the multiplication of some elements of $\discalg{A}$ to the discriminant form of $A$.

\begin{lemma}\label{relation-with-discriminant-form}
Let $R$ be a ring. Let $A$ be an $R$-algebra of rank $n\geq 2$. 
Let $a=(a_1,\ldots, a_n)$ and $b = (b_1,\ldots, b_n)$ be in $A^n$.
Then the equality
\[
\bigl(\dot\gamma^{\A{n}}(a) - \dot\gamma^{\Abar{n}}(a)\bigr)
\bigl(\dot\gamma^{\A{n}}(b) - \dot\gamma^{\Abar{n}}(b)\bigr) = 
\disc{A}(a_1\wedge \cdots \wedge a_n, b_1\wedge \cdots \wedge b_n)
\]
holds in $\discalg{A}$.
\end{lemma}
\begin{proof}
Recall fromm \cref{conjugate} at the beginning of \cref{section-understanding} the notation $\conjugate{(\cdot)}{i}$ for the $i$th embedding $A\to A^{\otimes n}$.
Then we may write the difference $\gamma^{\A{n}}(a) - \gamma^{\Abar{n}}(a)$ in $A^{\otimes n}$ as
\[
\sum_{\sigma\in\A{n}}\prod_{i=1}^n \conjugate{a_i}{\sigma(i)} - \sum_{\sigma\in\Abar{n}}\prod_{i=1}^n \conjugate{a_i}{\sigma(i)} = \sum_{\sigma \in 
\S{n}} \sgn(\sigma) \prod_{i=1}^n a_i^{(\sigma(i))},
\]
the determinant of the $n\times n$-matrix whose $ij$th element is 
$\conjugate{a_i}{j}$, and similarly $\gamma^{\A{n}}(b) - \gamma^{\Abar{n}}(b) 
= \det\bigl(\conjugate{b_k}{j}\bigr)_{jk}$. From this, we can write
\begin{align*}
\bigl(\gamma^{\A{n}}(a) - \gamma^{\Abar{n}}(a)\bigr)
\bigl(\gamma^{\A{n}}(b) - \gamma^{\Abar{n}}(b)\bigr) &=
\det\bigl(\conjugate{a_i}{j}\bigr)_{ij} \det \bigl(\conjugate{b_k}{j}\bigr)_{jk}\\
&= \det \bigl( \sum_{j=1}^n \conjugate{a_i}{j}\conjugate{b_k}{j} \bigr)_{ik}.
\end{align*}
The symbol $\conjugate{a_i}{j}\conjugate{b_k}{j}$ is equal to 
$\conjugate{(a_ib_k)}{j}$, and the sum over all $j$ gives by definition $e_1(a_i b_k)$. 
This is in $(A^{\otimes n})^{\S{n}}$ and by \cref{ek-sent-to-sk} its image under the Ferrand homomorphism is $s_1(a_i 
b_k) = \trace_A(a_i b_k)$. 
Hence in $\discalg{A}$ we have
\begin{align*}
\bigl(\dot\gamma^{\A{n}}(a) - \dot\gamma^{\Abar{n}}(a)\bigr)
\bigl(\dot\gamma^{\A{n}}(b) - \dot\gamma^{\Abar{n}}(b)\bigr) &= \det \bigl(\trace_A(a_ib_k)\bigr)_{ik}\\
&= \delta_A(a_1\wedge\dots\wedge a_n, b_1\wedge\dots\wedge b_n)
\end{align*}
as we wanted to show.
\end{proof}

For comparing the discriminant form of $\discalg{A}$ with that of $A$, we will need to understand the trace map $\trace_{\discalg{A}}\colon\discalg{A}\to R$. 
A helpful intermediate step is to understand the so-called standard involution on any rank-$2$ algebra:

\begin{definition}\label{std-involution}
Let $R$ be a ring and $D$ an $R$-algebra of rank $2$. 
The function $\stdi{D}\colon D\to D$ sending $d\mapsto \trace_D(d)-d$ is called the \emph{standard involution} on $D$, and has the property that
\[d+\stdi{D}(d) = \trace_D(d)\text{ and } d\cdot\stdi{D}(d) = \norm_D(d)\]
for all $d\in D$.
It is in fact an involution and an $R$-algebra homomorphism, and is even the unique $R$-algebra involution $\tau\colon D\to D$ such that $d\cdot\tau(d)\in R$ for all $d\in D$ (see, for example, \cite[Lemma 2.9]{Voight11}).
\end{definition}

\begin{corollary}\label{involution-on-Delta}
Let $R$ be a ring and $A$ an $R$-algebra of rank $n$, with $n \geq 2$. 
Let $\tau$ be the map $\discalg{A}\to\discalg{A}$ induced by the action of a transposition on the tensor factors in $(A^{\otimes n})^{\A{n}}$. 
Then $\tau$ is the standard involution on $\discalg{A}$. 
\end{corollary}
\begin{proof}
Since the action of a transposition on any $\S{n}$-invariant element is trivial, the map $\tau\colon \discalg{A}\to\discalg{A}$ is a well-defined $R$-algebra involution on $\discalg{A}$. 
And for any element $x$ of $\fixpower{A}{n}{\A{n}}$, we have that $x\cdot\tau(x)$ is $\S{n}$-invariant, and so is sent to an element of $R$ in $\discalg{A}$. 
Therefore $\tau$ is indeed the standard involution on $\discalg{A}$.
\end{proof}

We can now prove that the exact sequence of \cref{exact-sequence} identifies the discriminant forms of $A$ and $\discalg{A}$.

\begin{proof}[Proof of \cref{exact-sequence}(b)]
The composite isomorphism $\extpower^n A\cong\discalg{A}/R \cong \extpower^2\discalg{A}$ sends an element $a_1\wedge\cdots\wedge a_n$ of $\extpower^n A$ to $1\wedge \dot\gamma^{\A{n}}(a_1,\ldots, a_n)$ in
$\extpower^2\discalg{A}$. 
Given $a = (a_1,\ldots, a_n)$ and $b = (b_1,\ldots,b_n)$ in $A^n$ we have
\[
\delta_{\discalg{A}}\bigl(1\wedge\dot\gamma^{\A{n}}(a), 1\wedge\dot\gamma^{\A{n}} (b)\bigr) = \det \begin{pmatrix}
\trace_{\discalg{A}}(1) & \trace_{\discalg{A}}\bigl(\dot\gamma^{\A{n}}(b)\bigr) \\
\trace_{\discalg{A}}\bigl(\dot\gamma^{\A{n}}(a)\bigr) & \trace_{\discalg{A}}\bigl(\dot\gamma^{\A{n}}(a) \dot\gamma^{\A{n}}(b)\bigr)
\end{pmatrix}
\]
which is equal to
\[
2 \trace_{\discalg{A}}(\dot\gamma^{\A{n}}(a)\dot\gamma^{\A{n}}(b)) - \trace_{\discalg{A}}(\dot\gamma^{\A{n}}(a))\trace_{\discalg{A}}(\dot\gamma^{\A{n}}(b)).
\]
By \cref{involution-on-Delta}, the standard involution of $\discalg{A}$ is the one arising from the action of a transposition on the tensor factors of $A^{\otimes n}$, so we may compute these traces as sums:
\[
\begin{split}
\delta_{\discalg{A}} (1\wedge \dot\gamma^{\A{n}}(a), 1 \wedge
\dot\gamma^{\A{n}}(b)) &= 
2 \bigl(\dot\gamma^{\A{n}}(a) \dot\gamma^{\A{n}}(b) +
\dot\gamma^{\Abar{n}}(a)\dot\gamma^{\Abar{n}}(b)\bigr) \\
&\qquad - \bigl(\dot\gamma^{\A{n}}(a) + \dot\gamma^{\Abar{n}}(a)\bigr)
\bigl(\dot\gamma^{\A{n}}(b) + \dot\gamma^{\Abar{n}}(b)\bigr) \\
&= \bigl(\dot\gamma^{\A{n}}(a) - \dot\gamma^{\Abar{n}}(a)\bigr) 
\bigl(\dot\gamma^{\A{n}}(b) - \dot\gamma^{\Abar{n}}(b)\bigr)
\end{split}
\]
which is equal to $\delta_{A}(a_1\wedge \cdots \wedge a_n, b_1\wedge\cdots\wedge b_n)$ by \cref{relation-with-discriminant-form}. 
So the given isomorphism identifies the discriminant forms of $A$ and $\discalg{A}$, as we wanted to show.
\end{proof}

\section{Examples of discriminant algebras}\label{section-examples}

To show that computing $\discalg{A}$ for $A$ an algebra of rank $n$ is a 
straightforward process, in this section we will exhibit some examples of 
discriminant algebras and how to compute them. The simplest case, the 
discriminant algebra of a quadratic (i.e.\ rank-$2$) algebra, is reassuring but not very 
illuminating.

\begin{proposition}\label{discriminant-quad}
 Let $R$ be a ring and $A$ a quadratic $R$-algebra. 
 Then the $R$-algebra homomorphism $A\to \discalg{A}$ sending $a\mapsto\dot\gamma^{\A{2}}(1,a)$ is an isomorphism.
\end{proposition}
\begin{proof}
 That this is an $R$-algebra homomorphism is evident from its factorization as
 \[
 A\to A^{\otimes 2} = \fixpower{A}{2}{\A{2}} \to \discalg{A}
 \]
 sending $a\mapsto 1\otimes a=\gamma^{\A{2}}(1,a)\mapsto \dot\gamma^{\A{2}}(1,a)$. 
That it is an isomorphism then follows by the Five Lemma from the following commuting diagram:
 \[
\xymatrix{
0 \ar[r]& R \ar@{=}[d]\ar[r] & A \ar[d]\ar[r] & \extpower^2 A \ar@{=}[d] \ar[r] 
&
0 \\
0 \ar[r] & R \ar[r] & \discalg{A} \ar[r] & \extpower^2 A \ar[r] & 0,
}
\]

where the bottom row is the short exact sequence of \cref{exact-sequence}, and the map $A\to\extpower^2 A$ in the top row sends $a$ to $1\wedge a$. 
The left-hand square commutes because $A\to\discalg{A}$ is an $R$-algebra homomorphism, and the right-hand square commutes because the composite $A\to \discalg{A}\to\extpower^2 A$ also sends $a\mapsto\dot \gamma^{\A{2}}(1,a) \mapsto 1\wedge a$.
\end{proof}

The simplest non-trivial example of a discriminant algebra will therefore be that of a rank-$3$ algebra. 
The following proposition is quite useful for computing discriminant algebras of algebras that are free as modules.

\begin{proposition}\label{discriminant-free}
 Let $R$ be a ring, and suppose $A$ is an $R$-algebra that is free of rank $n\geq 2$ as an $R$-module, with $R$-basis $\theta=(\theta_1,\dots,\theta_n)$.  
 Then $\discalg{A}$ is also free as an $R$-module, with $R$-basis $\{1,\dot\gamma^{\A{n}}(\theta)\}$.
\end{proposition}
\begin{proof}
 Note that if $A$ is free with $R$-basis $(\theta_1,\dots,\theta_n)$, then $\extpower^n A$ is free of rank $1$, with generator $\theta_1\wedge\dots\wedge\theta_n$. 
 Then from the exact sequence of \cref{exact-sequence}, we find that $\{1,\dot\gamma^{\A{n}}(\theta)\}$ is an $R$-basis for $\discalg{A}$, because it is the disjoint union of an $R$-basis for $R$ and a lifting of an $R$-basis for $\extpower^n A$.
\end{proof}

\begin{remark}\label{discalg-generators}
 More generally, if $\Theta\subset A^n$ is a set of tuples such that
 \[\{\theta_1\wedge\dots\wedge \theta_n : (\theta_1,\dots,\theta_n)\in\Theta\}\]
 generates $\extpower^n A$ as an $R$-module, then $\discalg{A}$ is generated as an $R$-module by $1$ together with
 \[\{\dot\gamma^{\A{n}}(\theta_1,\dots,\theta_n) : (\theta_1,\dots,\theta_n)\in\Theta\}.\]
\end{remark}

\begin{remark}
In the setting of \cref{discriminant-free}, note that since $\dot\gamma^{\A{n}}(\theta)$ has trace $\dot\gamma^{\A{n}}(\theta)+\dot\gamma^{\Abar{n}}(\theta)$ and norm $\dot\gamma^{\A{n}}(\theta)\dot\gamma^{\Abar{n}}(\theta)$ by 
\cref{involution-on-Delta}, we find that 
\[
\discalg{A}\cong R[X]/(X^2 - (\dot\gamma^{\A{n}}(\theta)+\dot\gamma^{\Abar{n}}(\theta))X + (\dot\gamma^{\A{n}}(\theta)\dot\gamma^{\Abar{n}}(\theta))).
\]
Note that this quadratic polynomial in $X$ has discriminant
\begin{align*}
 \bigl(\dot\gamma^{\A{n}}(\theta)+\dot\gamma^{\Abar{n}}(\theta)\bigr)^2 - 4\bigl(\dot\gamma^{\A{n}}(\theta)\dot\gamma^{\Abar{n}}(\theta)\bigr) &= \bigl(\dot\gamma^{\A{n}}(\theta)-\dot\gamma^{\Abar{n}}(\theta)\bigr)^2 \\
  &= \disc{A}(\theta_1\wedge\dots\wedge\theta_n,\theta_1\wedge\dots\wedge\theta_n),
\end{align*}
the discriminant of $A$ with respect to the basis $\theta$.

As an immediate consequence, we find that the discriminant of $A$ with respect to the $R$-basis $\theta$ is a square in $R/(4)$.
This generalizes Stickelberger's theorem that the discriminant of a number field is always congruent to $0$ or $1$ modulo $4$.
\end{remark}

\begin{example}
Let $R$ be a ring, and let $E$ be a locally-free $R$-module of rank $n\geq 1$.
Then $R\oplus E$ can be given an $R$-algebra structure in which $(1,0)$ is the multiplicative identity and $E$ is a square-zero ideal.
Its discriminant algebra is then
\[\discalg{R\oplus E}\cong R\oplus\extpower^n E,\]
again with multiplicative identity $(1,0)$ and $\extpower^n E$ a square-zero ideal.

We show this by proving that there is a unique $R$-module homomorphism $\discalg{R\oplus E}\to R\oplus \extpower^n E$ sending $1\mapsto(1,0)$ and $\dot\gamma^{\A{n+1}}(1,e)\coloneqq \dot\gamma^{\A{n+1}}(1,e_1,\dots,e_n)\mapsto (0,e_1\wedge\dots\wedge e_n)$ for all tuples $e=(e_1,\dots,e_n)\in E^n$.
We will then show that this homomorphism is an isomorphism, and that the induced multiplication on $R\oplus \extpower^n E$ is the desired one.

Uniqueness is clear, as such elements generate $\discalg{R\oplus E}$ by \cref{discalg-generators}.
For existence, then, it suffices to check locally, when $E$ is free with basis $\theta = (\theta_1,\dots,\theta_n)$.
Then $\discalg{R\oplus E}$ is free with basis $\bigl(1, \dot\gamma^{\A{n+1}}(1,\theta)\bigr)$ and $R\oplus\extpower^n E$ is free with basis $\bigl((1,0), (0,\theta_1\wedge\dots\wedge\theta_n)\bigr)$, so we may define an $R$-module homomorphism $\discalg{R\oplus E}\to R\oplus \extpower^n E$ sending $1\mapsto (1,0)$ and 
\[\dot\gamma^{\A{n+1}}(1,\theta)\mapsto(0,\theta_1\wedge\dots\wedge\theta_n).\]
We claim that this homomorphism has the desired property, that
\[\dot\gamma^{\A{n+1}}(1,e)\mapsto (0,e_1\wedge\dots\wedge e_n)\]
for all $e = (e_1,\dots,e_n)\in E^n$.
We check this first in case $e = \theta_\sigma$, that is, the $e_i$ are a permutation of the $\theta_i$---see the discussion after \cref{multiplication-on-Delta}.
If $\sigma$ is even, then $\dot\gamma^{\A{n+1}}(1,e) = \dot\gamma^{\A{n+1}}(1,\theta)$ and $e_1\wedge\dots\wedge e_n = \theta_1\wedge\dots\wedge\theta_n$, and we are done.
If $\sigma$ is odd, then $\dot\gamma^{\A{n+1}}(1,e) = \dot\gamma^{\Abar{n+1}}(1,\theta)$, and what we instead must show is that
\[\dot\gamma^{\Abar{n+1}}(1,\theta)\mapsto(0, -\theta_1\wedge\dots\wedge \theta_n).\]
We show this by proving that the sum of $\dot\gamma^{\A{n+1}}(1,\theta)$ and $\dot\gamma^{\Abar{n+1}}(1,\theta)$ is zero, and so the image of the latter is the negative of the image of the former.

Now we can compute $\dot\gamma^{\A{n+1}}(1,\theta)+\dot\gamma^{\Abar{n+1}}(1,\theta) = \dot\gamma^{\S{n+1}}(1,\theta) = \ferrand_{A}\bigl(\gamma^{\S{n+1}}(1,\theta)\bigr)$ as the coefficient of $\lambda_0\lambda_1\cdots\lambda_n$ in the norm of $\lambda_0+\sum_{i=1}^n\lambda_i\theta_i$.
This element acts via the following matrix:
\[\begin{pmatrix}
 \lambda_0 & 0 & 0 & \dots & 0\\
 \lambda_1 & \lambda_0 & 0 & \dots & 0\\
 \lambda_2 & 0 & \ddots & \ddots & \vdots\\
 \vdots & \vdots & \ddots & \lambda_0 & 0\\
 \lambda_n & 0 & \dots & 0 & \lambda_0
\end{pmatrix}\]
This matrix is lower-triangular and has $\lambda_0$s on the diagonal, so the determinant is $\lambda_0^{n+1}$, in which the coefficient of $\lambda_0\lambda_1\cdots\lambda_n$ is zero, as desired.
Thus the homomorphism we have defined does indeed send $\dot\gamma^{\A{n+1}}(1,e)$ to $(0,e_1\wedge\dots\wedge e_n)$ whenever the $e_i$ are a permutation of the $\theta_i$.

Now we check that this holds for general $e$.
Write $e_i = \sum_{j} m_{ij} \theta_j$ for some $m_{ij}\in R$.
Then we have 
\begin{align*}
 \dot\gamma^{\A{n+1}}(1, e) &= \dot\gamma^{\A{n+1}}\Bigl(1,\sum_{j=1}^n m_{1j}\theta_j,\dots,\sum_{j=1}^n m_{nj}\theta_j\Bigr)\\
 &= \sum_{f\colon \set{n}\to\set{n}}\Bigl(\prod_{i=1}^n m_{if(i)}\Bigr)\dot\gamma^{\A{n+1}}(1, \theta_{f(1)},\dots,\theta_{f(n)}).
\end{align*}
We claim that the terms with $f$ not injective equal zero.
Indeed, if $f$ is not injective then $\gamma^{\A{n+1}}(1,\theta_{f(1)},\dots,\theta_{f(n)})$ is $\S{n+1}$-invariant and equals $\gamma^{\alpha}(1,\theta_1,\dots,\theta_n)$ for some $\alpha\in\set[0]{n}^{n+1}/\S{n+1}$.
But then its image in $\discalg{A}$ is the coefficient of $\lambda^{\alpha}$ in the norm of $\lambda_0 + \sum_{i=1}^n\lambda_i\theta_i$, which vanishes unless $\alpha = \{(0,0,\dots,0)\}$, but this case cannot arise.
Therefore
\begin{align*}
 \dot\gamma^{\A{n+1}}(1, e) &= \sum_{\sigma\in\S{n}}\Bigl(\prod_{i=1}^n m_{i\sigma(i)}\Bigr)\dot\gamma^{\A{n+1}}(1, \theta_{\sigma(1)},\dots,\theta_{\sigma(n)})\\
 &\mapsto(0, \sum_{\sigma\in\S{n}}\Bigl(\prod_{i=1}^n m_{i\sigma(i)}\Bigr)\sgn(\sigma)\ \theta_1\wedge\dots\wedge\theta_n)\\
 &= (0, \det\bigl(m_{ij}\bigr)_{ij}\ \theta_1\wedge\dots\wedge\theta_n)\\
 &= (0, e_1\wedge\dots\wedge e_n)
\end{align*}
as desired.
Thus we have an $R$-module homomorphism $\discalg{R\oplus E}\to R\oplus \extpower^n E$ sending $\dot\gamma^{\A{n+1}}(1,e)$ to $(0, e_1\wedge\dots\wedge e_n)$, and since it is an isomorphism locally, it is an isomorphism of $R$-modules.

Then all that remains to check is that the induced $R$-algebra structure on $R\oplus\extpower^n E$ is the indicated one. 
Now as long as $n\geq 2$, every term in a product of the form \[\dot\gamma^{\A{n+1}}(1,e_1,\dots,e_n)\dot\gamma^{\A{n+1}}(1,f_1,\dots,f_n)\] vanishes because every product $e_i f_j = 0$.
Thus the multiplication on $R\oplus\extpower^n E$ sets $(e_1\wedge\dots\wedge e_n)(f_1\wedge\dots\wedge f_n)=0$ as desired.
If we are in the case $n=1$, then $R\oplus E$ is a rank-$2$ algebra and we have $\discalg{R\oplus E}\cong R\oplus E \cong R\oplus\extpower^1 E$ anyway.
\end{example}

\begin{example}\label{discriminant-monogenic-cubic}
Let $R$ be a ring and $A$ be an $R$-algebra of rank $n$ that can be generated by a single element $a$.
Then if $p_a(\lambda)$ is the characteristic polynomial of $a$, we have $A\cong R[X]/(p_a(X))$. 
In particular, $\{1,a,\dots,a^{n-1}\}$ is an $R$-basis of $A$. 
So we find that $\{1,\dot\gamma^{\A{n+1}}(1,a,\dots,a^{n-1})\}$ is an $R$-basis of $\discalg{A}$.

For example, if $n=3$ and $p_a(X) = X^3-sX^2+tX-u$, then we have an $R$-basis for $\discalg{A}$ given by $\{1, \dot\gamma^{\A{3}}(1,a,a^2)\}$. 
We can compute the trace and norm of the generator $\dot\gamma^{\A{3}}(1,a,a^2)$ as follows: 
for the trace, we need only a small calculation
\begin{align*}
\trace_{\discalg{A}}(\dot\gamma^{\A{3}}(1,a,a^2)) &=  \ferrand_{A}(\gamma^{\S{3}}(1,a,a^2)) \\
&= \ferrand_{A}(e_1(a)e_2(a) - 3 e_3(a))\\
&= s_1(a)s_2(a) -3s_3(a)\\
&= st-3u.
\end{align*}
For the norm, we have
\begin{align*}
  \norm_{\discalg{A}}(\dot\gamma^{\A{3}}(1,a,a^2)) &= \dot\gamma^{\A{3}}(1,a,a^2) \dot\gamma^{\Abar{3}}(1,a,a^2)\\
  &= \dot\gamma^{\A{3}}(1,a^3,a^3) + \dot\gamma^{\A{3}}(a^2,a^2,a^2) + \dot\gamma^{\A{3}}(a,a,a^4).
\end{align*}
(The reader may check this expansion by hand, or appeal to \cref{multiplication-on-Delta}.)
We have $\dot\gamma^{\A{3}}(a^2,a^2,a^2) = 3u^2$, as this is just three times the norm of $a^2$. 
Moreover, we have that $\gamma^{\A{3}}(a,a,a^4)$ is equal to $e_3(a)\gamma^{\A{3}}(1,1,a^3)$ so that $\dot\gamma^{\A{3}}(a,a,a^4)=u\dot\gamma^{\A{3}}(1,1,a^3)$. 
Moreover $\dot\gamma^{\A{3}}(1,1,a^3)$ and $\dot\gamma^{\A{3}}(1,a^3,a^3)$ are equal to $s_1(a^3)$ and $s_2(a^3)$ 
respectively. 
Hence, we just need to compute $s_1(a^3)$ and $s_2(a^3)$. 
The action of $a^3$ with respect to the basis $(1,a,a^2)$ is given by the matrix 
\[
\begin{pmatrix}
  0 & 0 & u \\
  1 & 0 & -t \\
  0 & 1 & s
 \end{pmatrix}^3 = \begin{pmatrix}
  u & su & s^2u-tu \\
  -t & u-st & su-s^2t+t^2 \\
   s & s^2-t & u-2st+s^3
 \end{pmatrix},
\]
which has trace $s_1(a^3)=3u-3st+s^3$ and quadratic trace $s_2(a^3) = 3u^2-3stu+t^3$. 
So we obtain
 \begin{align*}
  \norm_{\discalg{A}}(\dot\gamma^{\A{3}}(1,a,a^2)) &= s_2(a^3) + 3u^2 + us_1(a^3)\\
  &= (3u^2 - 3stu + t^3) + 3u^2 + u(3u-3st+s^3)\\
  &= 9u^2-6stu+t^3+s^3u.
 \end{align*}
 Therefore we have the following algebra presentation:
\[
\discalg{A}\cong R[X]/(X^2 - (st-3u)X + (9u^2-6stu+t^3+s^3u)).
\]
\end{example}


\begin{example}
Consider the cyclotomic extension of rings of integers $\Z[\zeta]\cong\Z[x]/(x^4+x^3+x^2+x+1)$ over $\Z$ for $\zeta$ a primitive $5$th root of unity.
The discriminant of this $\Z$-algebra is 
\[
\left|
\begin{matrix}
\trace_K(1) & \trace_K(\zeta) & \trace_K(\zeta^2) & \trace_K(\zeta^3) \\
\trace_K(\zeta) & \trace_K(\zeta^2) & \trace_{K}(\zeta^3) & \trace_{K}(\zeta^4) \\
\trace_{K}(\zeta^2) & \trace_{K}(\zeta^3) & \trace_{K}(\zeta^4) & \trace_{K}(1) \\
\trace_{K}(\zeta^3) & \trace_{K}(\zeta^4) & \trace_{K}(1) & \trace_{K}(\zeta) \\
\end{matrix}
\right|
=
\left|\begin{matrix}
4 & -1 & -1 & -1 \\
-1 & -1 & -1 & -1 \\
-1 & -1 & -1 & 4 \\
-1 & -1 & 4 & -1 \\
\end{matrix}\right|
= 5^3.
\]
Adjoining a square root of the discriminant to $\Z$, we obtain a naive discriminant algebra $\discalgs{naive}{O_K/\Z} = \Z[x]/(x^2-125)$. 
But the discriminant of this quadratic $\Z$-algebra is $(0)^2-4(-125) = 2^2\cdot 5^3$, which is \emph{not} equivalent to that of $O_K/\Z$, since $2^2$ is not the square of a unit in $\Z$.

Let us see instead what our construction $\discalg{O_K/\Z}$ gives.
The $\Z$-basis $\{1,\zeta,\zeta^2,\zeta^3\}$ for $\Z[\zeta]$ gives a basis $\{1,\dot\gamma^{\A{4}}(1,\zeta,\zeta^2,\zeta^3)\}$ for the discriminant algebra.
We compute the trace and norm of $\dot\gamma^{\A{4}}(1,\zeta,\zeta^2,\zeta^3)$ to give a presentation for this algebra:
\begin{align*}
 \trace_{\discalg{\Z[\zeta]}}\bigl(\dot\gamma^{\A{4}}(1,\zeta,\zeta^2,\zeta^3)\bigr) &= \dot\gamma^{\A{4}}(1,\zeta,\zeta^2,\zeta^3) + \dot\gamma^{\Abar{4}}(1,\zeta,\zeta^2,\zeta^3) \\
 &= \ferrand_{\Z[\zeta]}\bigl(\gamma^{\S{4}}(1,\zeta,\zeta^2,\zeta^3)\bigr),
\end{align*}
which the algorithm in \cref{ek-generators} expresses as
\begin{gather*}
 \ferrand_{\Z[\zeta]}\bigl(e_1(\zeta)e_1(\zeta^2)e_1(\zeta^3) - e_1(\zeta)e_1(\zeta^5) - e_1(\zeta^2)e_1(\zeta^4) - 2e_2(\zeta^3) + e_1(\zeta^6)\bigr)\\
 \begin{aligned}
 &= s_1(\zeta)s_1(\zeta^2)s_1(\zeta^3) - s_1(\zeta)s_1(\zeta^5)  - s_1(\zeta^2)s_1(\zeta^4) - 2s_2(\zeta^3) + s_1(\zeta^6)\\
 &= (-1)(-1)(-1) - (-1)(4) - (-1)(-1) - 2(1) + (-1) \\
 & = -1 + 4 - 1 - 2 - 1 = -1.
\end{aligned}
\end{gather*}
The norm we compute with the help of \cref{multiplication-on-Delta}:
\begin{gather*}
\norm_{\discalg{\Z[\zeta]}}\bigl(\dot\gamma^{\A{4}}(1,\zeta,\zeta^2,\zeta^3)\bigr) = \dot\gamma^{\A{4}}(1,\zeta,\zeta^2,\zeta^3) \dot\gamma^{\Abar{4}}(1,\zeta,\zeta^2,\zeta^3) \\
\begin{array}{rclcl}
 =& & \dot\gamma^{\A{4}}(\zeta,\zeta,\zeta^4,\zeta) &+& \dot\gamma^{\A{4}}(\zeta,\zeta^2,1,\zeta^4) \\
 & +& \dot\gamma^{\A{4}}(\zeta,\zeta^3,\zeta^3,1) &+& \dot\gamma^{\A{4}}(\zeta^2,1,1,1)\\
 & +& \dot\gamma^{\A{4}}(\zeta^2,\zeta^2,\zeta^2,\zeta) &+& \dot\gamma^{\A{4}}(\zeta^2, \zeta^3, \zeta^4, \zeta^3)\\
 & +& \dot\gamma^{\A{4}}(\zeta^3,1,\zeta^3,\zeta) &+& \dot\gamma^{\A{4}}(\zeta^3, \zeta, 1, \zeta^3)\\
 & +& \dot\gamma^{\A{4}}(\zeta^3, \zeta^3, \zeta^2, \zeta^4) &+& \dot\gamma^{\A{4}}(\zeta^4, 1, \zeta^4, \zeta^4)\\
 & +& \dot\gamma^{\A{4}}(\zeta^4, \zeta, \zeta^2, 1) &+& \dot\gamma^{\A{4}}(\zeta^4, \zeta^2, \zeta^3, \zeta^3).
 \end{array}
\end{gather*}
Each term with a repeated entry is easy to expand in terms of the $s_k(\zeta^j)$; for example:
\[\dot\gamma^{\A{4}}(\zeta,\zeta,\zeta^4,\zeta) = s_4(\zeta)\dot\gamma^{\A{4}}(1,1,\zeta^3,1) = 3s_4(\zeta)s_1(\zeta^3) = 3(1)(-1) = -3\]
and
\begin{align*}
\dot\gamma^{\A{4}}(\zeta,\zeta^3,\zeta^3,1) &= s_1(\zeta)s_2(\zeta^3) - s_1(\zeta^3)s_1(\zeta^4) + s_1(\zeta^7) \\
&= (-1)(1) - (-1)(-1) + (-1) = -3.
\end{align*}
In fact, all the terms equal $-3$ except for $\dot\gamma^{\A{4}}(\zeta,\zeta^2,1,\zeta^4)$ and $\dot\gamma^{\A{4}}(\zeta^4,\zeta,\zeta^2,1)$, whose sum when multiplied by $s_4(\zeta)=1$ gives $\trace\bigl(\dot\gamma^{\A{4}}(\zeta^2,\zeta^3,\zeta,1)\bigr)=-1$.
Thus the trace and norm of $\dot\gamma^{\A{4}}(1,\zeta,\zeta^2,\zeta^3)$ are $-1$ and $10(-3) - 1 = -31$, so
\[
\discalg{O_K/\Z} \cong \Z[x]/(x^2+x - 31),
\]
which does have discriminant $(1)^2-4(-31)=125$ equal to that of $O_K$.
\end{example}

 While theoretically well-behaved, sometimes the coefficients in a presentation of the discriminant algebra are hard to interpret.

\begin{example}
 Let $A$ be the rank-$n$ $\Z$-algebra $\Z[x]/(x^n-1)$ for $n\geq 2$.
 Then as in \cref{discriminant-monogenic-cubic} the discriminant algebra $\discalg{A/\Z}$ has basis $\{1,\dot\gamma^{\A{n}}(1,x,\dots,x^{n-1})\}$, and we can present the algebra structure of $\discalg{A/\Z}$ if we know the trace and norm of $\dot\gamma^{\A{n}}(1,x,\dots,x^{n-1})$.
  The trace is equal to $\ferrand_{A/\Z}\bigl(\gamma^{\S{n}}(1,x,\dots,x^{n-1})\bigr)$, which as in \cref{ferrand-example} we can compute as the coefficient of $\lambda_1\dots\lambda_n$ in the norm of $\lambda_1 + \lambda_2x + \dots + \lambda_n x^{n-1}$, which equals the determinant
 \[\left|\begin{matrix}
  \lambda_1 & \lambda_n & \lambda_{n-1} & \dots & \lambda_2\\
  \lambda_2 & \lambda_1 & \lambda_n & \ddots  & \vdots \\
  \lambda_3 & \lambda_2 & \ddots & \ddots & \lambda_{n-1} \\
  \vdots & \ddots & \ddots & \lambda_1 & \lambda_n\\
  \lambda_n & \dots & \lambda_3 & \lambda_2 & \lambda_1
 \end{matrix}\right|.\]
 This coefficient is equal to the permanent of Schur's matrix $(\zeta^{ij})_{i,j=1}^n$ for $\zeta$ a primitive $n$th root of unity; see \cite{GrahamLehmer} for this and the following facts.
 
 For small values of $n$, the trace of $\dot\gamma^{\A{n}}(1,x,\dots,x^{n-1})$ is shown in the following table:
 \[\begin{array}{c|ccccccccccc}
  n & 2 & 3 & 4 & 5 & 6 & 7 & 8 & 9 & 10 & 11\dots\\
  \hline
  \trace & 0 & -3 & 0 & -5 & 0 & 105 & 0 & 81 & 0 & 6765\dots
 \end{array}\]
 For even $n$, the trace is always zero; thus $\dot\gamma^{\A{n}}(1,x,\dots,x^{n-1})^2$ is in $\Z$ and equals one fourth of the discriminant, so
 \[\discalg{A/\Z} \cong \Z[x]/(x^2-\frac14(-1)^{n/2}n^n),\]
 the ``naive'' discriminant algebra if one remembers the factor of $1/4$.
 But if $n$ is odd, then the trace of $\dot\gamma^{\A{n}}(1,x,\dots,x^{n-1})$ is nonzero.
\end{example}

\section{Proofs of \cref*{mainthm-base-change,mainthm-2-unit,mainthm-etale}}\label{section-discalg-props}

In this section we demonstrate proofs of the other three main theorems from the introduction. 
First, we show that \cref{mainthm-base-change} holds, so that the operation $(R,A)\mapsto \discalg{A/R}$ commutes with base change.

\begin{theorem}\label{base-change}
Let $R$ be a ring, and let $A$ be an $R$-algebra of rank $n$ with $n\geq 2$. 
Let $R'$ be an $R$-algebra, and let $A' = R'\otimes A$, so that $A'$ is an $R'$-algebra of rank $n$. 
The natural $R'$-algebra isomorphism $R'\otimes \fixpower{A}{n}{\A{n}}\to \fixpower[R']{A'}{n}{\A{n}}$ descends to an $R'$-algebra isomorphism $R'\otimes \discalg{A/R}\to \discalg{A'/R'}$.
\end{theorem}
\begin{proof}
We proved that the canonical homomorphism $R'\otimes \fixpower{A}{n}{\A{n}}\to \fixpower[R']{A'}{n}{\A{n}}$ is an isomorphism in \cref{fixpower-commutes-with-base-change}.
Then since the Ferrand homomorphism also commutes with base change, we obtain an isomorphism $R'\otimes \discalg{A/R}\cong \discalg{A'/R'}$ as desired.
\end{proof}

Second, we can immediately show that satisfying \cref{mainthm-2-unit}, i.e.\ the discriminant algebra having the right description when $2$ is a unit in the base ring, is a consequence of satisfying \cref{mainthm-identify-discriminants}, as a corollary of the following proposition:

\begin{proposition}
Let $R$ be a ring in which $2$ is invertible.
Let $D$ be an $R$-algebra of rank $2$.
Give the rank-$2$ $R$-module $R\oplus\extpower^2 D$ an $R$-algebra stucture by making $(1,0)$ the multiplicative identity and setting $(0,\xi)\cdot(0,\psi) = \frac14\disc{D}(\xi,\psi)$.
Then there is a unique isomorphism of $R$-algebras $D\simto R\oplus \extpower^2 D$ inducing the identity map on top exterior powers
\[\extpower^2 D \simto \extpower^2 (R\oplus \extpower^2 D) \cong (R\oplus \extpower^2 D)/R = \extpower^2 D;\]
the isomorphism is $d\mapsto (\frac12\trace_D(d), 1\wedge d)$.
\end{proposition}

\begin{proof}
 First we establish existence.
 Note that the isomorphism $D/R \simto \extpower^2 D$ sending the class of $d$ to $1\wedge d$ gives us an exact sequence
 \[0 \to R\to D\to \extpower^2 D \to 0,\]
 which is split by the homomorphism $D\to R\colon  d\mapsto \frac12\trace_D(d)$.
 This gives us the desired isomorphism of $R$-modules $D\cong R\oplus \extpower^2 D$ sending $d$ to $(\frac12\trace_D(d), 1\wedge d)$.
 Next we show that transporting the $R$-algebra structure from $D$ to $R\oplus\extpower^2 D$ gives the desired multiplication on $R\oplus\extpower^2 D$.
 The image of the multiplicative identity $1$ is $(\frac12\trace_D(1), 1\wedge 1) = (1,0)$, as desired.
 All that remains, then, is to check the multiplication on elements of the form $(0,\xi)$, that is, the images of trace-zero elements of $D$.
 Then let $d,e\in D$ have trace zero; we will show that $de = \frac14\disc{D}(1\wedge d,1\wedge e)$ in $D$.
 First, note that the squares of $d$, $e$, and $d+e$ are all in $R$: they are roots of their characteristic polynomials, which have no linear term.
 Therefore $de = \frac12((d+e)^2 - d^2 - e^2)$ is in $R$ as well.
 Then we may compute $\disc{D}(1\wedge d,1\wedge e)$ as
 \[
  \disc{D}(1\wedge d,1 \wedge e) = \det \begin{pmatrix}
  \trace_D(1) & \trace_D(d) \\
  \trace_D(e) & \trace_D(de)
  \end{pmatrix} = \det \begin{pmatrix}
  2 & 0 \\
  0 & 2de
 \end{pmatrix} = 4de,
\]
so we have $de = \frac14 \disc{D}(1\wedge d, 1\wedge e)$ as we wanted to show.
Thus the induced product on $R\oplus \extpower^2 D$ has $(0,\xi)\cdot(0,\psi) = \frac14\disc{D}(\xi,\psi)$, as claimed, and we have demonstrated existence of the desired algebra isomorphism.

Now we show that there is only one isomorphism with these properties.
It is sufficient to show that the only algebra automorphism of $D$ inducing the identity on $\extpower^2 D$ is itself the identity.
For this we work locally, so assume $D$ has a basis of the form $\{1,x\}$.
An automorphism of $D$ sends $x$ to an element $ux+r$ for some $r\in R$ and $u$ a unit of $R$.
If this is to descend to the identity on $\extpower^2 D$, we must have $u=1$.
And if this is an automorphism of $D$, we must have $\trace_D(x) = \trace_D(ux+r) = \trace_D(x+r) = \trace_D(x) + 2r$, so $2r=0$.
Since $2$ is a unit of $R$, we have $r=0$ and the automorphism is $\id_D$.
\end{proof}

\begin{corollary}\label{discalg-when-2-is-invertible}
 Let $R$ be a ring in which $2$ is a unit, and let $A$ be a rank-$n$ $R$-algebra.
 Put an $R$-algebra structure on $R\oplus \extpower^n A$ by making $(1,0)$ the multiplicative identity, and setting $(0,\xi)\cdot(0,\psi) =\frac14\disc{A}(\xi,\psi)$.
 Then there is a unique $R$-algebra isomorphism $\discalg{A}\cong R\oplus \extpower^n A$ such that the resulting isomorphism $\extpower^2\discalg{A} \cong \extpower^2(R\oplus\extpower^n A) \cong \extpower^n A$ is the one from \cref{exact-sequence}.
\end{corollary}

Note the extra factor of $\frac14$ in the definition of the multiplication, as compared with the statement of \cref{mainthm-2-unit} on page \pageref{mainthm-2-unit}. 
Since $2$ is a unit, the two multiplications yield isomorphic algebra structures on $R\oplus\extpower^n A$, but we see that the one given here is the one which arises naturally from the identification of the discriminant forms.
This is analogous to the factor of $\frac14$ one could introduce to the naive discriminant algebra $K[x]/(x^2-\frac14D)$, to make its discriminant $D$ instead of $4D$.

For the remainder of this section, we will suppose that $R$ is \emph{connected}, i.e.\ that it has exactly two idempotents $0$ and $1$,
so that we are working in the setting of \cref{mainthm-etale}.
A finite \'etale algebra over a connected ring $R$ is automatically an $R$-algebra of rank $n$ for some natural number $n$, and a $R$-algebra $A$ of rank $n$ is finite \'etale if and only if its discriminant form $\disc{A}\colon\extpower^n A\times\extpower^n A\to R$ is nondegenerate, in the sense that the induced $R$-module homomorphism 
\[
\extpower^n A\to \Hom(\extpower^n A,R)
\]
is an isomorphism. 
It is clear from \cref{mainthm-identify-discriminants}, then, that  $A$ is \'etale if and only if $\discalg{A/R}$ is \'etale, since their discriminant forms are isomorphic.
 
The significance of finite \'etale $R$-algebras is clarified by the following theorem, due to Grothendieck and proven in \cite[Ch. V, \sect7]{SGA1}:

\begin{theorem}\label{theorem-etale-equivalence}
Let $R$ be a connected ring equipped with a ring homomorphism to a separably closed field $K$. 
Then there is a profinite group $\fundgroup{R}$, called the \emph{fundamental group of $R$ (at $K$)}, such that for each finite \'etale $R$-algebra $A$, the finite set $F(A)=\Hom_{R\algs}(A,K)$ is naturally equipped with a continuous $\fundgroup{R}$-action. Furthermore, the assignment $A\mapsto F(A)$ defines a contravariant equivalence of categories
\[
F\colon R\finetalgs \to \fundgroup{R}\fingsets
\]
from the category of finite \'etale $R$-algebras to the category of $\fundgroup{R}$-sets, i.e.\ finite sets equipped with a continuous $\fundgroup{R}$-action.
\end{theorem}
 
Note: Grothendieck's original formulation was for the fundamental group of locally Noetherian schemes, but the Noetherian hypothesis is not necessary; 
an excellent reference is \cite[Section 5]{GaloisSchemes}. Note also that the fundamental group $\fundgroup{R}$ implicitly depends on the choice of $K$;
different choices of $K$ yield fundamental groups that are isomorphic but not canonically so, a behavior which is analogous to the dependence of the 
topological fundamental group on a choice of base point. 
In the following, we will suppress the dependence of $\pi_R$ on $K$ whenever possible without causing confusion.

\begin{example}
If $R=K$ is a field and $K^s$ is its separable closure, then $\fundgroup{K}$ is naturally identified with the absolute Galois group of $K$, which acts continuously on $\Hom_{K\algs}(A,K^s)$ for each finite separable $K$-algebra $A$.
\end{example}

\begin{remark}\label{etale-correspondence-remark}
In general, \cref{theorem-etale-equivalence} tells us quite a lot about the finite $\fundgroup{R}$-set corresponding to a given finite \'etale $R$-algebra: 
\begin{enumerate}
\item For each finite \'etale $R$-algebra $A$, the rank of $A$ is the 
cardinality of $F(A)$.
Indeed, suppose that $A$ is locally free of rank $n$ as an $R$-module. 
Then as sets, $F(A)= \Hom_{R\algs}(A,K) \cong \Hom_{K\algs}(K\otimes_R A,K)$, but since $K\otimes_R A$ is a finite separable $K$-algebra of rank $n$, it is isomorphic to $K^n$ and there are exactly $n$ $K$-algebra homomorphisms $K^n \to K$. 
(In fact, the choice of isomorphism $K\otimes_R A\cong K^n$ amounts to a choice of bijection $\set{n}\simto F(A)$, so $K\otimes_R A$ is canonically isomorphic to $K^{F(A)}$ as a $K$-algebra.)
\item If $A_1,\dots,A_n$ are finite \'etale $R$-algebras, then $F(\prod_{i=1}^n A_i)$ is isomorphic to the disjoint union $\coprod_{i=1}^n F(A_i)$, and $F(\bigotimes_{i=1}^n A_i)$ to the product $\prod_{i=1}^n F(A_i)$, each with the induced $\fundgroup{R}$-action. 
This is just because a contravariant equivalence sends limits to colimits and vice versa, and products and tensor products of \'etale algebras are \'etale.
\item In particular, the zero $R$-algebra corresponds to the empty $\fundgroup{R}$-set, and $R$ itself to a singleton $\{\ast\}$ with the trivial $\fundgroup{R}$-action. 
Then for any finite set $S$, the \emph{trivial} \'etale $R$-algebra $R^S=\prod_{s\in S}R$ corresponds to a $\fundgroup{R}$-set isomorphic to the set $S$ equipped with the trivial $\fundgroup{R}$-action.
\item If $G$ is a finite group acting via $R$-algebra isomorphisms on a finite \'etale $R$-algebra $A$, then $G$ also acts naturally on the corresponding $\fundgroup{R}$-set $F(A)$. 
Furthermore, the $R$-algebra of $G$-invariants $A^G$ is also finite \'etale, and the corresponding $\fundgroup{R}$-set $F(A^G)$ is isomorphic to $F(A)/G$, the set of $G$-orbits of $F(A)$. 
(See \cite[Prop.\ 5.20]{GaloisSchemes}.)
\end{enumerate}
\end{remark}

With these remarks, we will see that the Ferrand homomorphism has an especially nice interpretation in the case of \'etale algebras:

\begin{proposition}\label{Ferrand-map-maps-to-bijections}
Let $R$ be a connected ring equipped with a ring homomorphism to a separably closed field $K$. 
Let $\fundgroup{R}$ be the fundamental group. 
Let $A$ be a finite \'etale $R$-algebra corresponding to an $n$-element $\fundgroup{R}$-set $X$ under the equivalence of \cref{theorem-etale-equivalence}. 
Then the Ferrand homomorphism $\ferrand_{A}\colon \fixpower{A}{n}{\S{n}}\to R$ corresponds to the map of $\fundgroup{R}$-sets $\{\ast\}\to X^n/\S{n}$ sending $\ast$ to the $\S{n}$-orbit of bijections $\set{n}\simto X$.
\end{proposition}
\begin{proof}
By changing base (to $K$ if necessary) we may assume that $A$ is the trivial \'etale algebra $R^X$.
Then Ferrand shows in \cite[Ex.\ 3.1.3(b)]{FonctNorme} that the Ferrand homomorphism $\fixpower{(R^X)}{n}{\S{n}}\cong R^{X^n/\S{n}}\to R$ is projection onto the coordinate indexed by the $\S{n}$-orbit of bijections $\set{n}\simto X$.
\end{proof}

Finally, we prove the main theorem of the section.

\begin{theorem}\label{discriminant-algebra-corresponds-to-orientations}
Let $R$ be a connected ring with fundamental group $\fundgroup{R}$, and let $A$ be a finite \'etale $R$-algebra corresponding to an $n$-element $\fundgroup{R}$-set $X$, with $n\geq 2$. 
Then the discriminant algebra $\discalg{A}$ is a finite \'etale $R$-algebra corresponding to the $2$-element $\fundgroup{R}$-set
\[
\Or(X) \coloneqq  \Bij(\set{n}, X)/\A{n}
\]
of \emph{orientations} on $X$.
\end{theorem}
\begin{proof}
Note that $\discalg{A}$ is defined so that the following square is a tensor product diagram:
\[
\begin{tikzcd}
\discalg{A} & R \arrow{l}\\
(A^{\otimes n})^{\A{n}} \arrow{u}& (A^{\otimes n})^{\S{n}} \arrow{u}\arrow{l}
\end{tikzcd}
\]
Since $A$ is \'etale, this square is in fact a cofibered product in the category of finite \'etale $R$-algebras. Hence under the contravariant equivalence $F\colon R\finetalgs\to\fundgroup{R}\fingsets$, this diagram becomes a fiber product diagram of $\fundgroup{R}$-sets:
\[
\begin{tikzcd}
 F(\discalg{A}) \arrow{r} \arrow{d}& \{\ast\} \arrow{d}\\
X^n/\A{n} \arrow{r}& X^n/\S{n}
\end{tikzcd}
\]
This pullback can be computed in the category of sets, and then equipped canonically with a $\fundgroup{R}$-action. 
By \cref{Ferrand-map-maps-to-bijections} the image of $*$ in $X^n/\S{n}$ is the class of bijections $\Bij(\set{n},X)$. 
So the elements of $F(\discalg{A})$ are $\A{n}$-equivalence classes of maps $\set{n}\to X$, whose underlying $\S{n}$-equivalence class is $\Bij(\set{n},X)$. 
In other words, this is the set of $\A{n}$-equivalence classes of bijections $\set{n}\simto X$, namely the two-element set $\Or(X)$.
\end{proof}

\section{Functoriality}
\label{section-functoriality}

Given an isomorphism of rank-$n$ $R$-algebras $f\colon A\to B$, we obtain isomorphisms $\fixpower{f}{n}{\A{n}}\colon\fixpower{A}{n}{\A{n}}\to \fixpower{B}{n}{\A{n}}$ and $\fixpower{f}{n}{\S{n}}\colon\fixpower{A}{n}{\S{n}}\to \fixpower{B}{n}{\S{n}}$, and thus an isomorphism of tensor products $\discalg{f}\colon\discalg{A/R}\to\discalg{B/R}$.
The key here is that the following triangle of $R$-algebra homomorphisms commutes if $f$ is an isomorphism:
\begin{equation}
\label{ferrand-functorial}
\begin{tikzcd}[column sep=tiny]
   \fixpower{A}{n}{\S{n}} \arrow{rd}[swap]{\ds\ferrand_{A/R}}\arrow{rr}{\fixpower{f}{n}{\S{n}}} & & \fixpower{B}{n}{\S{n}} \arrow{ld}{\ds\ferrand_{B/R}}\\
   & R &
  \end{tikzcd}
\end{equation}

For a general homomorphism $f\colon A\to B$, if the triangle \eqref{ferrand-functorial} commutes, then we obtain a homomorphism of discriminant algebras $\discalg{f}\colon\discalg{A}\to\discalg{B}$.

\begin{proposition}\label{norm-preserving}
 Let $R$ be a ring, and $A$ and $B$ two $R$-algebras of rank $n$.
 Let $f\colon A\to B$ be an $R$-algebra homomorphism, and let $\Omega\subset A$ be a set of elements of $A$ whose powers generate $A$ as an $R$-module.
 Then the following are equivalent:
 \begin{enumerate}[label=(\roman*)]
  \item The triangle \eqref{ferrand-functorial} of $R$-algebra homomorphisms commutes.
  \item For all $a\in A$, we have $\norm_{A/R}(a) = \norm_{B/R}(f(a))$, and the same holds after base change to any $R$-algebra $R'$.
  \item For all $a\in \Omega$, the characteristic polynomial of $a$ equals the characteristic polynomial of $f(a)$.
 \end{enumerate}
\end{proposition}

\begin{proof}
 (i $\Rightarrow$ ii) After base changing to $R'$, we obtain a commuting triangle
 \[\begin{tikzcd}[column sep=tiny]
   \fixpower[R']{A'}{n}{\S{n}} \arrow{rd}[swap]{\ds\ferrand_{A'/R'}}\arrow{rr}{\fixpower{f'}{n}{\S{n}}} & & \fixpower[R']{B'}{n}{\S{n}} \arrow{ld}{\ds\ferrand_{B'/R'}}\\
   & R' &
  \end{tikzcd}\]
  One path sends $a\otimes\dots\otimes a$ to $\norm_{A'/R'}(a)$, and other sends it to $\norm_{B'/R'}(f'(a))$, so these must be equal.
  
 (ii $\Rightarrow$ iii) Given any $a\in A$, we have from (ii) applied to $R'=R[\lambda]$ the equation $\norm_{A[\lambda]/R[\lambda]}(\lambda-a) = \norm_{B[\lambda]/R[\lambda]}(\lambda-f(a))$, but these are exactly the characteristic polynomials of $a$ and $f(a)$.
 
 (iii $\Rightarrow$ i) The composite $\ferrand_{B/R}\circ\fixpower{f}{n}{\S{n}}$ sends $e_k(a)\mapsto e_k(f(a))\mapsto s_k(f(a)) = s_k(a)$ for each $a\in\Omega$, and therefore equals $\ferrand_{A/R}$ by \cref{ek-generators}.
\end{proof}

We call an $R$-algebra homomorphism \emph{universally norm-preserving} if it satisfies the equivalent conditions of \cref{norm-preserving}.
Isomorphisms are always universally norm-preserving, but the following example shows that universally norm-preserving homomorphisms are in general neither injective nor surjective.

\begin{example}
 Let $A$ be any rank-$n$ $R$-algebra and $a\in A$ any element, with characteristic polynomial $p_a(\lambda) = \norm_{A[\lambda]/R[\lambda]}(\lambda-a)$.
 Then the $R$-algebra homomorphism $R[x]/(p_a(x))\to A$ sending $x$ to $a$ is universally norm-preserving, since $\Omega = \{x\}\subset R[x]/(p_a(x))$ is a set whose powers $\{1,x,x^2,\dots\}$ generate $R[x]/(p_a(x))$ as an $R$-module, and the characteristic polynomial of $x$ in $R[x]/(p_a(x))$ is again $p_a(\lambda)$.
 
 For example, let $r$ and $s$ be two elements of any ring $R$.
 Then we have a universally norm-preserving homomorphism $R[x]/(x-r)(x-s)\to R^2$ sending $x\mapsto (r,s)$.
 This map is injective if and only if $r-s$ is not a zerodivisor in $R$, and surjective if and only if $r-s$ is a unit.
 
 Note that there are also well-defined homomorphisms $R[x]/(x-r)(x-s)\to R^2$ sending $x$ to $(r,r)$ or $(s,s)$, but these are not universally norm-preserving unless $r=s$.
\end{example}

\begin{remark}
 The $\S{n}$-closure operation defined by Bhargava and Satriano in \cite{BhargSat} has a universal property related to universally norm-preserving homomorphisms.
 Namely, a \emph{full set of sections} for a rank-$n$ algebra $A$ is a family of $n$ algebra homomorphisms $A\to R$ such that the resulting homomorphism of rank-$n$ algebras $A\to R^n$ is universally norm-preserving, and the $\S{n}$-closure of $A$ over $R$ is the universal $R$-algebra over which $A$ is equipped with a full set of sections.
\end{remark}

%
%
%

\section{The discriminant algebra of a product}\label{section-products}

If $A$ and $B$ are $R$-algebras of ranks $m$ and $n$, respectively, then their product $A\times B$ is an $R$-algebra of rank $m+n$. 
It is well-known that $\extpower^{m+n}(A\times B)$ is canonically isomorphic to $\extpower^m A \otimes \extpower^n B$, and that if we identify these two locally-free $R$-modules we find that the discriminant quadratic form $\disc{A\times B}$ is just $\disc{A}\otimes\disc{B}$.
In this sense, the construction of the discriminant form is multiplicative.
We extend this multiplicativity to the construction of the discriminant algebra $\discalg{A\times B}$.

Namely, for each ring $R$ there is a commutative monoid structure $\ast_R$ on $\Quad_R$, the set of isomorphism classes of quadratic $R$-algebras; 
these monoid structures have recently been characterized by Voight in \cite{Voight15} as the unique family of functions $\ast_R\colon \Quad_R\times\Quad_R\to\Quad_R$ such that
\begin{itemize}
 \item For each $R$, the set $\Quad_R$ is a commutative monoid with multiplication $\ast_R$ and unit the class of $R^2$,
 \item For each $R$-algebra $R'$, the base-change operation $\Quad_R\to\Quad_{R'}$ is a homomorphism of commutative monoids, and
 \item If $S$ and $T$ are quadratic \'etale $R$-algebras with standard involutions $\sigma$ and $\tau$, then $S\ast_R T$ is the class of the subring of $S\otimes T$ fixed by $\sigma\otimes\tau$.
\end{itemize}
In \cite{DeligneLett}, Deligne asserted the existence of such a binary operation over which the discriminant algebra should distribute; 
an explicit construction of this binary operation on quadratic algebras is due to Loos and can be found in \cite{LoosTensorProd}.

Our goal in this section is to show that $\discalg{}$ is multiplicative in the sense that $\discalg{A\times B}\cong \discalg{A}\ast \discalg{B}$.
Our approach will be to show first that the operation on quadratic algebras $(S,T)\mapsto \discalg{S\times T}$ satisfies the three properties of the operations $\ast_R$ listed above, and then to exhibit an isomorphism $\discalg{A\times B}\cong \discalg{\discalg{A}\times\discalg{B}}$.
Then we will have
\[\discalg{A\times B}\cong \discalg{\discalg{A}\times\discalg{B}}\cong \discalg{A}\ast\discalg{B}.\]

\begin{remark}
The easiest properties to check are base-change and commutativity, and the description in the \'etale case is also straightforward to verify.
Indeed, if $S$ and $T$ are quadratic $R$-algebras, then $S\times T\cong T\times S$ so by functoriality under isomorphisms we find that
\[\discalg{S\times T}\cong \discalg{T\times S}.\] 
Furthermore, if $R'$ is any $R$-algebra, then
\[R'\otimes_R \discalg{(S\times T)/R} \cong \discalg{R'\otimes (S\times T)/R'} \cong \discalg{(S'\times T')/R'},\]
so the operation sending $(S,T)$ to the isomorphism class of $\discalg{S\times T}$ commutes with base change.

For the \'etale case, it is sufficient to check in case the base ring $R$ is connected.
(Indeed, in the proof of uniqueness in \cite{Voight15} the \'etale description is only used in a single universal case for which the base ring is even a domain.)
Then if we suppose $S$ and $T$ to correspond to $\fundgroup{R}$-sets $X$ and $Y$, to check that $(S\otimes T)^{\S{2}}$ is isomorphic to $\discalg{S\times T}$ we need only find an isomorphism of $\fundgroup{R}$-sets
\[(X\times Y)/\S{2} \cong \Bij\bigl(\set{4},X\sqcup Y\bigr)/\A{4}.\]
To wit, we have isomorphisms $X\cong\Bij(\{1,2\},X)$ and $Y\cong\Bij(\{3,4\},Y)$, and an inclusion
\[\Bij(\{1,2\},X)\times\Bij(\{3,4\},Y)\into \Bij\bigl(\set{4},X\sqcup Y\bigr).\]
It is easy to check that after projecting to $\Bij\bigl(\set{4},X\sqcup Y\bigr)/\A{4}$, we obtain a well-defined bijection from $\bigl(\Bij(\{1,2\},X)\times\Bij(\{3,4\},Y)\bigr)/\S{2}$.
\end{remark}

Checking that the operation $(S,T)\mapsto \discalg{S\times T}$ is associative and has unit $R^2$, and that there is an isomorphism $\discalg{A\times B}\cong \discalg{\discalg{A}\times\discalg{B}}$ for general $A$ and $B$, will occupy the rest of this section.
To check these statements, we will have to dig into explicit computations with elements.

We use the following abuse of notation in the hope that it will prevent the reader from being buried by parentheses. 
In product algebras of the form $A\times B$, we will often need to refer to elements of the form $(a,0)$ or $(0,b)$. 
We will usually denote such elements simply by $a$ or $b$, thus implicitly identifying the rings $A$ and $B$ with their corresponding ideals $A\times 0$ and $0\times B$ in $A\times B$. 
Since each of $A$ and $B$ naturally contains an image of $R$, for each $r\in R$ we will write $r_A$ for $(r,0)$ and $r_B$ for $(0,r)$; 
thus the two idempotents $(1,0)$ and $(0,1)$ in $A\times B$ are $1_A$ and $1_B$, respectively, and their sum is the unit $1$.

\begin{remark}\label{extpower-product-R}
For example, in this notation the isomorphism 
\[\extpower^m A\otimes \extpower^n B\cong \extpower^{(m+n)}(A\times B)\] is the one sending $(a_1\wedge\dots\wedge a_m)\otimes (b_1\wedge\dots\wedge b_{n})\mapsto a_1\wedge\dots\wedge a_m\wedge b_1\wedge\dots\wedge b_n$ for all $a_1,\dots,a_m\in A$ and $b_1,\dots,b_n\in B$ (see \cite[Prop.\ 10 on p.\ III.84]{BourbakiAlg13}).
\end{remark}
\begin{remark}
\label{delta-product-generators}
As a consequence, $\discalg{A\times B}$ is generated as an $R$-module by $1$ together with elements of the form $\dot\gamma^{\A{m+n}}(a_1,\dots,a_m,b_1,\dots,b_n)$ with $a_1,\dots,a_m\in A$ and $b_1,\dots,b_n\in B$ by \cref{discalg-generators}.
We will usually write such an element as $\dot\gamma^{\A{m+n}}(a,b)$, where $a=(a_1,\dots,a_m)\in A^m$ and $b=(b_1,\dots,b_n)\in B^n$.
\end{remark}

We will also often use the following description of the Ferrand homomorphism for $A\times B$:

\begin{lemma}\label{product-properties}
 Let $R$ be a ring, and let $A$ and $B$ be $R$-algebras of ranks $m$ and $n$, respectively.  
 Identify $\S{m}$ and $\S{n}$ with subgroups of $\S{m+n}$ permuting the first $m$ and last $n$ elements of $\set{m+n}$, respectively. 
 \begin{enumerate}[ref=\cref{product-properties}(\arabic*)]
 \item \label{product-algebra-invariants}
 For each subgroup $G\subset \S{m+n}$, the $R$-algebra projection
\[
\pi\colon(A\times B)^{\otimes m+n}\to A^{\otimes m}\otimes B^{\otimes n}
\]
sending $(a_1,b_1)\otimes\dots\otimes (a_{m+n},b_{m+n})$ to $(a_1\otimes\dots \otimes a_m)\otimes (b_{m+1}\otimes\dots\otimes b_{m+n})$ restricts to an $R$-algebra homomorphism
\[
\fixpower{(A\times B)}{(m+n)}{G}\to \fixpower{A}{m}{G\cap \S{m}} 
\otimes \fixpower{B}{n}{G\cap\S{n}}.
\]
 \item \label{lemma-ferrand-product}
 In case $G=\S{m+n}$, the resulting composite
 \[
 \begin{tikzcd}
 \fixpower{(A\times B)}{(m+n)}{\S{m+n}}\to \fixpower{A}{m}{\S{m}} 
\otimes \fixpower{B}{n}{\S{n}}\arrow{r}{\ferrand_{A}\otimes\ferrand_{B}} & R\otimes R\cong R
 \end{tikzcd}
 \]
 is the Ferrand homomorphism $\ferrand_{A\times B}$.
 \end{enumerate}
\end{lemma}

\begin{proof}
(1) Set $H=G\cap \S{m}$ and $K=G\cap \S{n}$. 
Then $H\times K$ is naturally a subgroup of $G$, and we have an inclusion
\[
\fixpower{(A\times B)}{(m+n)}{G}\into\fixpower{(A\times B)}{(m+n)}{H\times K}.
\]
Thus it is enough to show that $\pi\colon(A\times B)^{\otimes(m+n)}\to A^{\otimes m}\otimes B^{\otimes n}$ restricts to a homomorphism $\fixpower{(A\times B)}{(m+n)}{H\times K}\to \fixpower{A}{m}{H}\otimes\fixpower{B}{n}{K}$. 
Now $\S{m}\times\S{n}$ acts on both $(A\times B)^{\otimes m+n}$ and $A^{\otimes m}\otimes B^{\otimes n}$ by permuting the first $m$ and last $n$ tensor factors separately, and the projection $\pi$ is equivariant with respect to this action. 
Thus we obtain a map of $H\times K$-invariants
\[
\fixpower{(A\times B)}{(m+n)}{H\times K}\to \left(A^{\otimes m}\otimes 
B^{\otimes n}\right)^{H\times K}.
\]
Finally, observe that $\left(A^{\otimes m}\otimes B^{\otimes n}\right)^{H\times K}=\fixpower{A}{m}{H}\otimes\fixpower{B}{n}{K}$ as subalgebras of $A^{\otimes m}\otimes B^{\otimes n}$. 
(By \cref{fixpower-commutes-with-base-change}, this can be checked after changing base to a localization in which both $A$ and $B$ are free $R$-modules, in which case the isomorphism is elementary.) 
Thus we have obtained the desired homomorphism as the composite
\[
\fixpower{(A\times B)}{(m+n)}{G}\into\fixpower{(A\times B)}{(m+n)}{H\times 
K}\to\fixpower{A}{m}{H}\otimes\fixpower{B}{n}{K}.
\]

(2) We check the defining property of $\ferrand_{A\times B}$.
Let $(a,b)$ be an arbitrary element of $A\times B$; is $(a,b)\otimes\dots\otimes (a,b)$ sent to $\norm_{A\times B}(a,b) = \norm_A(a)\norm_B(b)$?
Indeed so: we have
\begin{align*}
(a,b)\otimes\dots\otimes(a,b) &\mapsto (a\otimes\dots\otimes a)\otimes(b\otimes\dots\otimes b) \\
&\mapsto (\norm_A(a))\otimes(\norm_B(b))\\
&\mapsto \norm_A(a)\norm_B(b)
\end{align*}
as desired.
The same property holds also after base change, so the indicated homomorphism is $\ferrand_{A\times B}$.
\end{proof}

Our first application is to check that $R^2$ is a unit for the operation $(S,T)\mapsto \discalg{S\times T}$.
More generally, if $A$ is any rank-$n$ $R$-algebra, then we have the following $R$-algebra isomorphism $\discalg{R\times A}\cong\discalg{A}$:

%
%
\begin{theorem}\label{base-factor}
Let $R$ be a ring, and let $A$ be an $R$-algebra of rank $n\geq 2$. 
The $R$-algebra homomorphism $(R\times A)^{\otimes (n+1)}\to A^{\otimes n}$ sending \[(r_0,a_0)\otimes(r_1,a_1)\otimes\dots\otimes(r_n,a_n) \mapsto r_0\cdot (a_1\otimes\dots\otimes a_n)\] restricts to an $R$-algebra homomorphism $\fixpower{(R\times A)}{(n+1)}{\A{n+1}}\to \fixpower{A}{n}{\A{n}}$ that descends to an $R$-algebra isomorphism $\discalg{R\times A} \to \discalg{A}$.
\end{theorem}
\begin{proof}
The homomorphism $(R\times A)^{\otimes n+1}\to A^{\otimes n}\cong R^{\otimes 1}\otimes A^{\otimes n}$ is of the form considered in \cref{product-algebra-invariants}, with $R$ and $A$ in place of $A$ and $B$.
Using the subgroup $\A{n+1}\subset\S{n+1}$, whose intersection with $\S{n}$ is $\A{n}$, \cref{product-algebra-invariants} tells us that the homomorphism restricts to one
\[
\fixpower{(R\times A)}{(n+1)}{\A{n+1}}\to R\otimes \fixpower{A}{ n}{\A{n}}\cong\fixpower{A}{n}{\A{n}},
\]
as desired.  
In addition, we know from \cref{lemma-ferrand-product} that the further restriction $\fixpower{(R\times A)}{(n+1)}{\S{n+1}} \to \fixpower{A}{n}{\S{n}}$ commutes with the Ferrand homomorphisms to $R$; 
thus we obtain a morphism of tensor products
\[
\mathop{\fixpower{(R\times A)}{(n+1)}{\A{n+1}} \midotimes R}_{\qquad\qquad\qquad \fixpower{(R\times A)}{(n+1)}{\S{n+1}}}\to\mathop{\fixpower{A}{n}{\A{n}} \midotimes 
R}_{\qquad\quad \fixpower{A}{n}{\S{n}}},
\]
i.e.\ an $R$-algebra homomorphism $\discalg{R\times A}\to\discalg{A}$.

To show that this homomorphism is an isomorphism, we show that it fits into a map of short exact sequences
\[
\xymatrix{
0 \ar[r]&
R \ar[r] \ar@{=}[d]&
\discalg{R\times A} \ar[r] \ar[d]&
\extpower^{n+1}(R\times A) \ar[r] &
0 \\
0 \ar[r]&
R \ar[r]&
\discalg{A} \ar[r]&
\extpower^n A \ar[r]\ar[u]^{\sim}&
0
}
\]where the right-hand isomorphism is the one $\extpower^n A\cong \extpower^1R\otimes\extpower^n A\cong \extpower^{n+1} (R\times A)$ from \cref{extpower-product-R}. 
The left-hand square commutes because $\discalg{R\times A}\to\discalg{A}$ is an $R$-algebra homomorphism. 
To show that the right-hand square commutes, follow an element $\dot\gamma^{\A{n+1}}(1_R,a_1,\dots,a_n)$ of $\discalg{R\times A}$; 
such elements along with $1$ generate $\discalg{R\times A}$ by \cref{delta-product-generators}.
Its image in $\discalg{A}$ is $\dot\gamma^{\A{n}}(a_1,\dots,a_n)$, giving $a_1\wedge\dots\wedge a_n$ in $\extpower^n A$, thence $1_R\wedge a_1\wedge\dots\wedge a_n$ in $\extpower^{n+1}(R\times A)$.
But this corresponds exactly to the action of the homomorphism $\discalg{R\times A} \to \extpower^{n+1}(R\times A)$.
Thus the right-hand square commutes as well, so by the Five Lemma, the homomorphism $\discalg{R\times A}\to\discalg{A}$ is an isomorphism of $R$-algebras.
\end{proof}

\begin{remark}
Certain authors (such as Deligne in \cite{DeligneLett} and Loos in \cite{LoosDiscAlg}) construct a discriminant algebra for only even-rank or only odd-rank algebras, defining the discriminant algebra of an algebra $A$ whose rank is of the wrong parity to be that of $R\times A$. 
The fact that our construction is invariant under adding a factor of $R$ will be useful in future work comparing these different constructions.
\end{remark}

\begin{remark}
 Note that we thus have an isomorphism $\discalg{A}\cong\discalg{R^2\times A}$ for all $R$-algebras of rank at least $2$, but the right-hand side is also well-defined for $R$-algebras of rank $0$ or $1$, and is then isomorphic to $R^2$.
 In this way we can extend the domain of the discriminant algebra operation to all constant rank (and even locally constant rank) algebras.
\end{remark}

All that remains, then, to show that $\discalg{S\times T}\cong S\ast T$ is to check that the operation $(S,T)\mapsto \discalg{S\times T}$ is associative.
This will follow from our exhibiting an isomorphism $\discalg{A\times B}\cong \discalg{\discalg{A}\times\discalg{B}}$, for then
\[
\discalg{S\times\discalg{T\times U}}
\cong\discalg{\discalg{S}\times\discalg{T\times U}}
\cong\discalg{S\times(T\times U)}
\cong\discalg{(S\times T)\times U}
\cong\discalg{\discalg{S\times T}\times\discalg{U}}
\cong\discalg{\discalg{S\times T}\times U},
\]
where we have used the isomorphisms $S\cong\discalg{S}$ and $U\cong\discalg{U}$ from \cref{discriminant-quad}.
So without any more ado, here is the promised isomorphism:

\begin{theorem}\label{main-discalg-prod}
Let $R$ be a ring, and let $A$ and $B$ be $R$-algebras of ranks $m$ and $n$, respectively, with both $m$ and $n$ at least $2$. 
Then there is a unique $R$-algebra isomorphism $\discalg{A\times B}\simto \discalg{\discalg{A}\times\discalg{B}}$ that sends $1\mapsto 1$ and $\dot\gamma^{\A{m+n}}(a,b)$ to $\dot\gamma^{\A{4}}\bigl(1_{\discalg{A}},\dot\gamma^{\A{m}}(a),1_{\discalg{B}},\dot\gamma^{\A{n}}(b)\bigr)$ for each $a\in A^m$ and $b\in B^n$.
\end{theorem}

The proof that this assignment describes a well-defined algebra isomorphism is long and unenlightening.
The authors hope that future work will reveal a simpler construction of this isomorphism that avoids the elementary slog to follow.
Indeed, for the proof we use the following lemma several times:

\begin{lemma}\label{gamma-product-identities}
Let $a=(a_1,\dots,a_m)\in A^m$ and $b=(b_1,\dots,b_n)\in B^n$.
Then
\[
\ferrand_{A\times B}\bigl(\gamma^{\S{m+n}}(a,b)\bigr) = \ferrand_{A}\bigl(\gamma^{\S{m}}(a)\bigr)\cdot \ferrand_{B}\bigl(\gamma^{\S{n}}(b)\bigr).
\]
If two of the $a_i$ are equal, then $\gamma^{\A{m+n}}(a,b)$ is $\S{m+n}$-invariant, and
\[
\ferrand_{A\times B}\bigl(\gamma^{\A{m+n}}(a,b)\bigr) = \ferrand_{A}\bigl(\gamma^{\A{m}}(a)\bigr)\cdot \ferrand_{B}\bigl(\gamma^{\S{n}}(b)\bigr).
\]
\end{lemma}

\begin{proof}
 By \cref{lemma-ferrand-product}, the Ferrand homomorphism for $A\times B$ may be computed by first projecting to $\fixpower{A}{m}{\S{m}}\otimes\fixpower{B}{n}{\S{n}}$ and then applying $\ferrand_A\otimes\ferrand_B$.
 In our case, the terms of $\gamma^{\S{m+n}}(a,b)$ that survive after the projection are exactly those with the $a_i$ among the first $m$ tensor factors and the $b_i$ among the last $n$.
 Thus the image of $\gamma^{\S{m+n}}(a,b)$ in $\fixpower{A}{m}{\S{m}}\otimes\fixpower{B}{n}{\S{n}}$ is $\gamma^{\S{m}}(a)\otimes\gamma^{\S{n}}(b)$.
 Therefore
 \[
\ferrand_{A\times B}\bigl(\gamma^{\S{m+n}}(a,b)\bigr) = \ferrand_{A}\bigl(\gamma^{\S{m}}(a)\bigr)\cdot \ferrand_{B}\bigl(\gamma^{\S{n}}(b)\bigr)
\]
 as desired.
 
 If two of the $a_i$ are equal, then the image of $\gamma^{\A{m+n}}(a,b)$ after projecting to $A^{\otimes m}\otimes B^{\otimes n}$ is the sum $\gamma^{\A{m}}(a)\otimes \gamma^{\A{n}}(b) + \gamma^{\Abar{m}}(a)\otimes \gamma^{\Abar{n}}(b)$, which equals $\gamma^{\A{m}}(a)\otimes\gamma^{\S{n}}(b)$ since $\gamma^{\A{m}}(a) = \gamma^{\Abar{m}}(a)$. 
Thus we obtain
\[
\ferrand_{A\times B}\bigl(\gamma^{\A{m+n}}(a,b)\bigr) = \ferrand_{A}\bigl(\gamma^{\A{m}}(a)\bigr)\cdot \ferrand_{B}\bigl(\gamma^{\S{n}}(b)\bigr)
\]
as claimed.
\end{proof}
\begin{example}\label{gamma-product-identities-example}
As a special case, consider quadratic $R$-algebras $S$ and $T$ with $s\in S$ and $t\in T$. 
We have
\begin{equation}
\begin{aligned}
\ferrand_{S\times T}\bigl(\gamma^{\A{4}}(s,s,1_T,t)\bigr) &= \ferrand_{S}\bigl(\gamma^{\A{2}}(s,s)\bigr)\cdot \ferrand_{T}\bigl(\gamma^{\S{2}}(1,t)\bigr) \\
 &= \ferrand_{S}(s\otimes s) \cdot \ferrand_{T}(1\otimes t + t\otimes 1)\\
 &= \norm_{S}(s) \cdot \trace_{T}(t).
\end{aligned}
\end{equation}
Similarly, we have
\begin{equation}
\begin{aligned}
\ferrand_{S\times T}\bigl(\gamma^{\A{4}}(s,s,t,t)\bigr) &= \ferrand_{S}\bigl(\gamma^{\A{2}}(s,s)\bigr)\cdot \ferrand_{T}\bigl(\gamma^{\S{2}}(t,t)\bigr) \\
 &= \ferrand_{S}(s\otimes s) \cdot \ferrand_{T}(t\otimes t + t\otimes t)\\
 &= \norm_{S}(s) \cdot 2\norm_{T}(t)\\
 &= 2\norm_{S}(s)\norm_{T}(t).
\end{aligned}
\end{equation}
These two identities will come up again in the proof of \cref{main-discalg-prod}.
\end{example}

\begin{proof}[Proof of \cref{main-discalg-prod}]
For uniqueness, note that $\discalg{A\times B}$ is generated as an $R$-module by $1$ and elements of the form $\dot\gamma^{\A{m+n}} (a_1,\dots,a_m,b_1,\dots,b_n)$ by \cref{delta-product-generators}. 
We demonstrate the existence of such an isomorphism on each localization of $R$ under which $A$ and $B$ become free $R$-modules. 
By uniqueness, then, these isomorphisms on the localizations will glue to an $R$-algebra isomorphism between the two original discriminant algebras.
 
Thus it suffices to assume $A$ and $B$ are free, say with $R$-bases $\theta=(\theta_1,\dots,\theta_m)$ and $\phi=(\phi_1,\dots,\phi_n)$, respectively. 
In that case, $\discalg{A\times B}$ is freely generated as an $R$-module by $1$ and $\dot\gamma^{\A{m+n}} (\theta_1,\dots,\theta_m,\phi_1,\dots,\phi_n)$; 
we abbreviate the latter as $\dot\gamma^{\A{m+n}}(\theta,\phi)$.
Similarly, $\discalg{\discalg{A}\times\discalg{B}}$ is freely generated as an $R$-module by $1$ and $\dot\gamma^{\A{4}}\bigl(1_{\discalg{A}}, \dot\gamma^{\A{m}}(\theta),1_{\discalg{B}},\dot\gamma^{\A{n}} (\phi)\bigr)$. 
Then we can naively define an $R$-module isomorphism
\[
f \colon \discalg{A\times B}\to \discalg{\discalg{A}\times\discalg{B}}\\
\]
sending $1 \mapsto 1$ and $\dot\gamma^{\A{m+n}}(\theta,\phi) \mapsto \dot\gamma^{\A{4}}\bigl(1_{\discalg{A}}, \dot\gamma^{\A{m}}(\theta),1_{\discalg{B}},\dot\gamma^{\A{n}}(\phi)\bigr)$.
We claim that in fact
\[
f \colon \dot\gamma^{\A{m+n}}(a,b) \mapsto \dot\gamma^{\A{4}}\bigl(1_{\discalg{A}}, \dot\gamma^{\A{m}}(a),1_{\discalg{B}},\dot\gamma^{\A{n}}(b)\bigr)
\]
for all $a=(a_1,\dots,a_m)\in A^m$ and $b=(b_1,\dots,b_n)\in B^n$, so that $f$ acts elementwise as desired. 
Then we will show that $f$ is multiplicative, making it an $R$-algebra isomorphism.
 
Since the expressions $\dot\gamma^{\A{m+n}}(a,b)$ and $\dot\gamma^{\A{4}}\bigl(1_{\discalg{A}}, \dot\gamma^{\A{m}}(a),1_{\discalg{B}},\dot\gamma^{\A{n}}(b)\bigr)$ 
are each multilinear in the $a_i$ and $b_j$, we can reduce to proving the claim in case each of the $a_i$ and $b_j$ are among the basis elements for $A$ and $B$.

The first possibility is that two of the $a_i$ or two of the $b_j$ are equal. 
Without loss of generality, assume that the $a_i$ are not all distinct. 
Then $\gamma^{\A{m+n}}(a,b)$ is $\S{m+n}$-invariant and $\dot\gamma^{\A{m+n}}(a,b)$ is equal to $\ferrand_{A\times B}\bigl(\gamma^{\A{m+n}}(a,b)\bigr) = \ferrand_{A}\bigl(\gamma^{\A{m}}(a)\bigr)\ferrand_{B}\bigl(\gamma^{\S{n}}(b)\bigr)$ by \cref{gamma-product-identities}.

On the other hand, since $\gamma^{\A{m}}(a)\in\fixpower{A}{m}{\A{m}}$ is also $\S{m}$-invariant, we can also express $\dot\gamma^{\A{4}}\bigl(1_{\discalg{A}}, \dot\gamma^{\A{m}}(a),1_{\discalg{B}},\dot\gamma^{\A{n}}(b)\bigr)$ in a form amenable to the \cref{gamma-product-identities-example} identities:
\begin{align*}
\dot\gamma^{\A{4}}\bigl(1_{\discalg{A}}, \dot\gamma^{\A{m}}(a),1_{\discalg{B}},\dot\gamma^{\A{n}}(b)\bigr) &= \dot\gamma^{\A{4}}\Bigl(1_{\discalg{A}}, \ferrand_{A}\bigl(\gamma^{\A{m}}(a)\bigr)_{\discalg{A}},1_{\discalg{B}},\dot\gamma^{\A{n}}(b)\Bigr)\\
&= \ferrand_{A}\bigl(\gamma^{\A{m}}(a)\bigr) \dot\gamma^{\A{4}}\bigl(1_{\discalg{A}}, 1_{\discalg{A}},1_{\discalg{B}},\dot\gamma^{\A{n}}(b)\bigr)\\
&= \ferrand_{A}\bigl(\gamma^{\A{m}}(a)\bigr) \norm_{\discalg{A}}(1)\trace_{\discalg{B}}\bigl(\dot\gamma^{\A{n}}(b)\bigr)\\
&= \ferrand_{A}\bigl(\gamma^{\A{m}}(a)\bigr) \ferrand_{B}\bigl(\gamma^{\S{n}}(b)\bigr).
\end{align*}
Therefore in this case $f$ sends $\dot\gamma^{\A{m+n}}(a,b)$ to $\dot\gamma^{\A{4}}\bigl(1_{\discalg{A}}, \dot\gamma^{\A{m}}(a),1_{\discalg{B}},\dot\gamma^{\A{n}}(b)\bigr)$ 
because these two elements are the same $R$-multiple of $1$.

The second possibility is that the $a_i$ are a permutation of the $\theta_i$, and the $b_j$ are a permutation of the $\phi_j$. 
Write $a=\theta_\sigma$ and $b=\phi_\tau$ for appropriate permutations $\sigma\in\S{m}$ and $\tau\in\S{n}$, recalling the notation from \cref{multiplication-on-Delta} for the action of the symmetric group on the set of tuples.
We must show that
\[
f\colon\dot\gamma^{\A{m+n}}(\theta_\sigma,\phi_\tau)\mapsto \dot\gamma^{\A{4}}\bigl(1_{\discalg{A}}, \dot\gamma^{\A{m}}(\theta_\sigma),1_{\discalg{B}},\dot\gamma^{\A{n}}(\phi_\tau)\bigr).
\]
There are four cases, according to the signs of $\sigma$ and $\tau$, which determine the values of $\dot\gamma^{\A{m+n}}(\theta_\sigma,\phi_\tau)$, $\dot\gamma^{\A{m}}(\theta_\sigma)$, and $\dot\gamma^{\A{n}}(\phi_\tau)$. 
So we must show that the following four assignments hold:
\begin{align*}
f\colon\dot\gamma^{\A{m+n}}(\theta,\phi)&\mapsto \dot\gamma^{\A{4}}\bigl(1_{\discalg{A}}, \dot\gamma^{\A{m}}(\theta),1_{\discalg{B}},\dot\gamma^{\A{n}}(\phi)\bigr)\text{, from $\sigma$ and $\tau$ even,}\\
f\colon\dot\gamma^{\Abar{m+n}}(\theta,\phi)&\mapsto \dot\gamma^{\A{4}}\bigl(1_{\discalg{A}}, \dot\gamma^{\Abar{m}}(\theta),1_{\discalg{B}},\dot\gamma^{\A{n}}(\phi)\bigr)\text{, from $\sigma$ odd and $\tau$ even,}\\
f\colon\dot\gamma^{\A{m+n}}(\theta,\phi)&\mapsto \dot\gamma^{\A{4}}\bigl(1_{\discalg{A}}, \dot\gamma^{\Abar{m}}(\theta),1_{\discalg{B}},\dot\gamma^{\Abar{n}}(\phi)\bigr)\text{, from $\sigma$ and $\tau$ odd, and}\\
f\colon\dot\gamma^{\Abar{m+n}}(\theta,\phi)&\mapsto \dot\gamma^{\A{4}}\bigl(1_{\discalg{A}}, \dot\gamma^{\A{m}}(\theta),1_{\discalg{B}},\dot\gamma^{\Abar{n}}(\phi)\bigr)\text{, from $\sigma$ even and $\tau$ odd.}
\end{align*}

The first of the four assignments holds by the definition of $f$. 
As for the others, note that 
\begin{align*}  
\dot\gamma^{\A{m+n}}(\theta,\phi)+\dot\gamma^{\Abar{m+n}}(\theta,\phi)&= \dot\gamma^{\S{m+n}}(\theta,\phi)\\
&= \ferrand_{A\times B}\left(\gamma^{\S{m+n}}(\theta,\phi)\right)\\
&= \ferrand_{A}\left(\gamma^{\S{m}}(\theta)\right)\ferrand_{B}\left(\gamma^{\S{n}}(\phi)\right),
\end{align*}
so it is enough to show that the sum of each successive pair of outputs is $\ferrand_{A}\left(\gamma^{\S{m}}(\theta)\right)\ferrand_{B} \left(\gamma^{\S{n}}(\phi)\right)$. 
And this indeed holds: for example, 
\begin{gather*}
\dot\gamma^{\A{4}}\bigl(1_{\discalg{A}}, \dot\gamma^{\A{m}}(\theta),1_{\discalg{B}},\dot\gamma^{\A{n}}(\phi)\bigr)+\dot\gamma^{\A{4}}\bigl(1_{\discalg{A}}, \dot\gamma^{\Abar{m}}(\theta),1_{\discalg{B}},\dot\gamma^{\A{n}}(\phi)\bigr)\\
\begin{aligned}
&= \dot\gamma^{\A{4}}\bigl(1_{\discalg{A}}, \dot\gamma^{\A{m}}(\theta)+\dot\gamma^{\Abar{m}}(\theta),1_{\discalg{B}},\dot\gamma^{\A{n}}(\phi)\bigr)\\
&= \dot\gamma^{\A{4}}\bigl(1_{\discalg{A}}, \dot\gamma^{\S{m}}(\theta),1_{\discalg{B}},\dot\gamma^{\A{n}}(\phi)\bigr)\\
&= \dot\gamma^{\A{4}}\left(1_{\discalg{A}}, \ferrand_{A}\left(\gamma^{\S{m}}(\theta)\right)_{\discalg{A}},1_{\discalg{B}},\dot\gamma^{\A{n}}(\phi)\right)\\
&= \ferrand_{A}\left(\gamma^{\S{m}}(\theta)\right)\dot\gamma^{\A{4}}\left(1_{\discalg{A}}, 1_{\discalg{A}},1_{\discalg{B}},\dot\gamma^{\A{n}}(\phi)\right)\\
&= \ferrand_{A}\left(\gamma^{\S{m}}(\theta)\right) \norm_{\discalg{A}}(1)\,\trace_{\discalg{B}}\left(\dot\gamma^{\A{n}}(\phi)\right)\\
&= \ferrand_{A}\left(\gamma^{\S{m}}(\theta)\right)\ferrand_{B}\left(\gamma^{\S{n}}(\phi)\right).
\end{aligned}
\end{gather*}
The other two sums can be evaluated similarly. So indeed, we have shown that $f$ must send $\dot\gamma^{\A{m+n}}(\theta_\sigma,\phi_\tau)$ to $\dot\gamma^{\A{4}}\bigl(1_{\discalg{A}}, \dot\gamma^{\A{m}}(\theta_\sigma),1_{\discalg{B}},\dot\gamma^{\A{n}}(\phi_\tau)\bigr)$, and established the claim that
\[
f \colon \dot\gamma^{\A{m+n}}(a,b) \mapsto \dot\gamma^{\A{4}}\bigl(1_{\discalg{A}}, \dot\gamma^{\A{m}}(a),1_{\discalg{B}},\dot\gamma^{\A{n}}(b)\bigr)
\]
for all $a\in A^m$ and $b\in B^n$.

Now we show that the $R$-module isomorphism $f\colon\discalg{A\times B} \to\discalg{\discalg{A}\times\discalg{B}}$ is in fact an $R$-algebra isomorphism. 
Since $f$ is $R$-linear and $\{1,\dot\gamma^{\A{m+n}}(\theta,\phi)\}$ forms an $R$-basis for $\discalg{A\times B}$, all that we need to check is whether 
\begin{align*}
f(1\cdot1) &= f(1)\cdot f(1),\\
f(1\cdot \dot\gamma^{\A{m+n}}(\theta,\phi)) &= f(1)\cdot f(\dot\gamma^{\A{m+n}}(\theta,\phi)),\text{ and }\\
f(\dot\gamma^{\A{m+n}}(\theta,\phi)\cdot\dot\gamma^{\A{m+n}}(\theta,\phi)) &= f(\dot\gamma^{\A{m+n}}(\theta,\phi))\cdot f(\dot\gamma^{\A{m+n}}(\theta,\phi)).
\end{align*}
The first two hold because $f(1)=1$ by definition, so all that is left to check is that $f\left(\dot\gamma^{\A{m+n}}(\theta,\phi)\right)^2=f\left(\dot\gamma^{\A{m+n}}(\theta,\phi)^2\right).$
Now we can use \cref{multiplication-on-Delta} to expand this product:
\begin{align*}
\dot\gamma^{\A{m+n}}(\theta,\phi)^2
 &= \sum_{\sigma\in\A{m+n}}\dot\gamma^{\A{m+n}}\bigl((\theta,\phi)(\theta,\phi)_\sigma\bigr)\\
 &= \sum_{\substack{(\sigma,\tau)\in\\
\mathclap{(\A{m}\times\A{n})\cup(\Abar{m}\times\Abar{n})}}} \dot\gamma^{\A{m+n}}(\theta\theta_{\sigma},\phi\phi_{\tau}),
\end{align*}
since $(\theta,\phi)(\theta,\phi)_\sigma$ has a zero entry unless $\sigma$ belongs to $\A{m+n}\cap(\S{m}\times \S{n}) = (\A{m}\times\A{n})\cup(\Abar{m}\times\Abar{n})$.
Therefore
\begin{align*}
\quad f\left(\dot\gamma^{\A{m+n}}(\theta,\phi)^2\right) &= \sum_{\substack{(\sigma,\tau)\in\\\mathclap{(\A{m}\times\A{n})\cup(\Abar{m}\times\Abar{n})}}}\dot\gamma^{\A{4}}\bigl(1_{\discalg{A}},\dot\gamma^{\A{m}}(\theta\theta_\sigma),1_{\discalg{B}},\dot\gamma^{\A{n}}(\phi\phi_\tau)\bigr)\\
&=  \dot\gamma^{\A{4}}\bigl(1_{\discalg{A}}, \dot\gamma^{\A{m}}(\theta)^2,1_{\discalg{B}},\dot\gamma^{\A{n}}(\phi)^2\bigr)\\
&\qquad+ \dot\gamma^{\A{4}}\bigl(1_{\discalg{A}}, \dot\gamma^{\A{m}}(\theta)\dot\gamma^{\Abar{m}}(\theta),1_{\discalg{B}},\dot\gamma^{\A{n}}(\phi)\dot\gamma^{\Abar{n}}(\phi)\bigr).
\end{align*}
On the other hand,
\begin{align*}
f\left(\dot\gamma^{\A{m+n}}(\theta,\phi)\right)^2 &=  \dot\gamma^{\A{4}}\bigl(1_{\discalg{A}}, \dot\gamma^{\A{m}}(\theta),1_{\discalg{B}},\dot\gamma^{\A{n}}(\phi)\bigr)^2\\
&= \dot\gamma^{\A{4}}\bigl(1_{\discalg{A}}, \dot\gamma^{\A{m}}(\theta)^2,1_{\discalg{B}},\dot\gamma^{\A{n}}(\phi)^2\bigr)\\
&\qquad+ \dot\gamma^{\A{4}}\bigl(\dot\gamma^{\A{m}}(\theta), \dot\gamma^{\A{m}}(\theta),\dot\gamma^{\A{n}}(\phi),\dot\gamma^{\A{n}}(\phi)\bigr),
\end{align*}
by a similar application of \cref{multiplication-on-Delta}. 
One of each of the two terms on the right-hand sides matches immediately, but the others are also equal:
\begin{gather*}
\dot\gamma^{\A{4}}\bigl(1_{\discalg{A}}, \dot\gamma^{\A{m}}(\theta)\dot\gamma^{\Abar{m}}(\theta),1_{\discalg{B}},\dot\gamma^{\A{n}}(\phi)\dot\gamma^{\Abar{n}}(\phi)\bigr)\\
\begin{aligned}
&= \dot\gamma^{\A{4}}\left(1_{\discalg{A}}, \norm_{\discalg{A}}\bigl(\dot\gamma^{\A{m}}(\theta)\bigr)_{\discalg{A}},1_{\discalg{B}},\norm_{\discalg{B}}\bigl(\dot\gamma^{\A{n}}(\phi)\bigr)_{\discalg{B}}\right)\\
&= \norm_{\discalg{A}}\bigl(\dot\gamma^{\A{m}}(\theta)\bigr)\,\norm_{\discalg{B}}\bigl(\dot\gamma^{\A{n}}(\phi)\bigr)\,\dot\gamma^{\A{4}}\left(1_{\discalg{A}}, 1_{\discalg{A}},1_{\discalg{B}},1_{\discalg{B}}\right)\\
&= 2\,\norm_{\discalg{A}}\bigl(\dot\gamma^{\A{m}}(\theta)\bigr)\,\norm_{\discalg{B}}\bigl(\dot\gamma^{\A{n}}(\phi)\bigr)\\
&=\dot\gamma^{\A{4}}\bigl(\dot\gamma^{\A{m}}(\theta), \dot\gamma^{\A{m}}(\theta),\dot\gamma^{\A{n}}(\phi),\dot\gamma^{\A{n}}(\phi)\bigr).
\end{aligned}
\end{gather*}
Thus $f$ is an $R$-algebra isomorphism $\discalg{A\times B}\to\discalg{\discalg{A}\times\discalg{B}}$, as desired.
\end{proof}

%
%
%
\begin{remark}
The isomorphism of \cref{main-discalg-prod} does not entirely commute with the isomorphisms interchanging $A$ and $B$.
While it is always the case that
\[
\dot\gamma^{\A{4}}\bigl(1_{\discalg{A}}, \dot\gamma^{\A{m}}(a_1,\dots,a_m), 1_{\discalg{B}}, \dot\gamma^{\A{n}}(b_1,\dots,b_n)\bigr)
\]
equals
\[
\dot\gamma^{\A{4}}\bigl(1_{\discalg{B}}, \dot\gamma^{\A{n}}(b_1,\dots,b_n), 1_{\discalg{A}}, \dot\gamma^{\A{m}}(a_1,\dots,a_m)\bigr),
\]
the elements $\dot\gamma^{\A{m+n}}(a_1,\dots,a_m,b_1,\dots,b_n)$ and $\dot\gamma^{\A{m+n}}(b_1,\dots,b_n,a_1,\dots,a_m)$ are generally only equal if $m$ or $n$ is even; 
otherwise they are conjugates. 
Thus the chain of isomorphisms
\[
\discalg{A\times B} \cong \discalg{\discalg{A}\times\discalg{B}} \cong\discalg{\discalg{B}\times\discalg{A}} \cong \discalg{B\times A} \cong \discalg{A\times B}
\]
is the identity if the rank of $A$ or $B$ is even, but is the standard 
involution on $\discalg{A\times B}$ if the ranks of $A$ and $B$ are both odd.
This is precisely the behavior described by Deligne at the end of 
\cite{DeligneLett}.
\end{remark}

\bibliographystyle{acm}
\bibliography{RefList}

\end{document}